\documentclass{amsart}%
\usepackage{amsfonts}
\usepackage{amsmath}
\usepackage{amssymb}
\usepackage{graphicx}%
\newtheorem{theorem}{Theorem}[section]
\theoremstyle{plain}
\newtheorem{corollary}[theorem]{Corollary}
\newtheorem{lemma}[theorem]{Lemma}
\newtheorem{proposition}[theorem]{Proposition}
\newtheorem{definition}[theorem]{Definition}
\theoremstyle{remark}
\newtheorem{remark}[theorem]{Remark}
\numberwithin{equation}{section}

\begin{document}
\title[Exceptional circles]{Exceptional circles of radial potentials}
\author{Michael Music}
\address{Michael Music, Department of Mathematics\\
University of Kentucky, Lexington, Kentucky 40506--0027}
\email{michael.music@uky.edu}
\author{Peter A. Perry}
\address{Peter A. Perry, Department of Mathematics\\
University of Kentucky, Lexington, Kentucky 40506--0027}
\email{perry@ms.uky.edu}
\author{Samuli Siltanen}
\address{Samuli Siltanen, University of Helsinki, P.O. Box 68, Helsinki, Finland FI-00014}
\email{samuli.siltanen@helsinki.fi}
\date{October 25, 2012}

\begin{abstract}
A nonlinear scattering transform is studied for the two-dimensional Schr\"{o}%
\-%
dinger equation at zero energy with a radial potential. First explicit examples are presented, both theoretically and computationally, of potentials with nontrivial
singularities in the scattering transform. The singularities arise from
non-uniqueness of the complex geometric optics solutions that define the
scattering transform. The values of the complex spectral parameter at which the singularities appear are called {\em exceptional points}. The singularity formation is closely related to the fact
that potentials of conductivity type are ``critical'' in the sense of Murata.

\end{abstract}
\maketitle


\section{Introduction}

\noindent
We study singularities of the scattering
transform at zero energy for two-dimensional Schr\"{o}dinger operators with radial and compactly supported potentials. We will present what is, to our knowledge,
the first example of a family of potentials for which the singularities of the
scattering transform can be computed explicitly. The singularities, occurring at so-called {\em exceptional points} of the potential, arise from
non-uniqueness of complex geometric optics (CGO) solutions of the
Schr\"{o}dinger equation. We also present the numerical computations which inspired this work and display the formation of
singularities under perturbation of a potential of conductivity type.

To motivate our main results, let us define a class of potentials that plays a central role in our work.

\begin{definition}\label{def:condtype}
A compactly supported real-valued potential $q\in C^\infty_0(\mathbb{R}^2)$ is {\em of conductivity type} if $q=\psi^{-1}\left(  \Delta\psi\right) $ for some real-valued $\psi\in C^\infty(\mathbb{R}^2)$ satisfying $\psi(z)\geq c > 0$ for all $z$ in a bounded set $\Omega\subset\mathbb{R}^2$ and $\psi(z) \equiv 1$  for all $z\in\mathbb{R}^2\setminus\Omega$.
\end{definition}
\noindent
 Note that the positive function $\psi$ above solves $\left(  -\Delta+q\right)  \psi=0$. This terminology arose when
Schr\"{o}dinger scattering theory was used to analyze the inverse conductivity
problem in Nachman \cite{Nachman:1996}, and was also needed in \cite{LMSS:2011,Perry:2012}. In those works $q$ is not necessarily compactly supported, but a condition implying $\lim_{\left\vert z\right\vert \rightarrow \infty}\psi(z)=1$ is crucial. In Appendix \ref{appendix:spectrum}, we show that the associated Schr\"{o}dinger operator has no eigenvalues, and that the function $\psi$ representing $q$ is unique.

Recall that the inverse conductivity problem of Calder\'on \cite{ca} consists in reconstructing the conductivity $\sigma$ of a
conducting body $\Omega\subset\mathbb{R}^{2}$ from the Dirichlet to Neumann
map, defined as follows. Let $f\in H^{1/2}(\partial\Omega)$ and let $u\in
H^{1}(\Omega)$ solve%
\begin{align*}
\nabla\cdot\left(  \sigma\nabla u\right)    & =0\\
\left.  u\right\vert _{\partial\Omega}  & =f.
\end{align*}
This solution is unique if $\sigma\in L^\infty(\Omega)$ is real-valued and strictly positive. The Dirichlet-to-Neumann map is the map
$\Lambda_{\sigma}:H^{1/2}(\partial\Omega)\rightarrow H^{-1/2}(\partial\Omega)$
given by%
\[
\Lambda_{\sigma}f=\sigma\left.  \frac{\partial u}{\partial\nu}\right\vert
_{\partial\Omega}.
\]
Calder\'on's problem is to reconstruct the function $\sigma$ from knowledge of
the map $\Lambda_{\sigma}$. In dimension two, Calder\'on's problem in its original form was solved in \cite{ap}.

Nachman \cite{Nachman:1996} exploited the fact that $\psi=\sigma^{1/2}u$
solves the Schr\"{o}dinger equation $\left(  -\Delta+q\right)  \psi=0$ where
$q=\sigma^{-1/2}\Delta\left(  \sigma^{1/2}\right)  $. The Schr\"{o}dinger
problem $\left(  -\Delta+q\right)  \psi=0$ also has a Dirichlet-to-Neumann map
$\Lambda_{q}$ which can be determined from $\Lambda_{\sigma}$. Using
$\Lambda_{q}$, Nachman was able to construct CGO
solutions to the Schr{\"o}dinger equation, use these solutions to compute the
scattering transform of $q$, and reconstruct $\sigma$ by inverting the scattering
transform. 

The theory of the scattering transform has been worked out in
detail for potentials of conductivity type (see Lassas, Mueller and Siltanen
\cite{LMS:2007} and references therein), but far less is known about general
classes.
Grinevich and Novikov \cite[part I of supplement 1]{GN:1988} commented that Schr\"odinger potentials with well-behaved scattering transforms are ``not in general position,'' and Nachman proved under minimal hypotheses that the scattering transform is regular if and only if the potential is of conductivity type. Below we show how work of Murata \cite{Murata:1986} implies that the set of conductivity type potentials is unstable under $\mathcal{C}_0^\infty$ perturbations and is therefore unstable in any reasonable function space!

Let us define the CGO\ solutions and scattering transform for a potential
$q\in\mathcal{C}_{0}^{\infty}(\mathbb{R}^{2})$. To define the CGO\ solutions,
let $k\in\mathbb{C}$, set $z=x+iy$ and write
\begin{equation}
kz=\left(  k_{1}+ik_{2}\right)  (x+iy)\label{eq:kx}%
\end{equation}
(complex multiplication). The CGO\ solutions $\psi(z,k)$ solve%
\begin{align}
\left(  -\Delta+q\right)  \psi &  =0,\label{eq:CGO}\\
\lim_{\left\vert z\right\vert \rightarrow\infty}e^{-ikz}\psi(z,k) &
=1.\nonumber
\end{align}
The set of \emph{nonzero }$k\in\mathbb{C}$ for which (\ref{eq:CGO}) does not have a unique
solution is called the \emph{exceptional set} $\mathcal{E}$. If the
exceptional set is empty,  the solutions $\psi(z,k)$ form a smooth family, and
the \emph{scattering transform} of $q$ is  given by%
\begin{equation}
\mathbf{t}(k)=\int e^{i\overline{k}\overline{z}}q(z)\psi(z,k)~dz.\label{eq:t}%
\end{equation}
The behavior of $\mathbf{t}$ at $k=0$ plays a special role which we will elucidate in what follows.

Nachman showed (under rather less stringent regularity assumptions than
$q\in\mathcal{C}_{0}^{\infty}(\mathbb{R}^{2})$) that $q$ is of conductivity type
if and only if:
\begin{enumerate}
\item[(i)] $\mathcal{E}$ is empty, and
\item[(ii)]
$\left| \mathbf{t}(k) \right| \leq C|k|^\epsilon$ for $|k|$ small and some $\epsilon>0$.
\end{enumerate} 
(see \cite{Nachman:1996}, Theorem 3). Until this time
it was not clear how the scattering transform behaved for potentials outside
this limited class. Our purpose here is to construct and analyze examples for
which the singularities of $\mathbf{t}(k)$ can be computed explicitly. In our
examples, $\mathbf{t}(k)$ is well-defined except on a circle in the complex
$k$-plane, and we can compute the singularity explicitly.

We will study the scattering transform for families of radial potentials
$q_{\lambda}$ defined as follows. Denote by $B_{1}$ the open unit disc in
$\mathbb{R}^{2}$ centered at $0$, so that $\partial B_{1}=S^{1}$ regarded as
an embedded manifold in $\mathbb{R}^{2}$. Suppose that $\sigma\in \mathcal{C}^\infty(\mathbb{R}^2)$ is a real-valued radial function satisfying $\sigma(z)>0$ for all $z\in\mathbb{R}^2$ and $\sigma-1\in\mathcal{C}_{0}^{\infty}\left(  B_{1}\right)$. Then $q_{0}:=\sigma^{-1/2}\Delta\sigma^{1/2}\in
\mathcal{C}_{0}^{\infty}\left(  B_{1}\right)$ is a radial potential of
conductivity type. For a nonnegative radial function $w\in
\mathcal{C}_{0}^{\infty}(B_{1})$ not identically zero, set%
\begin{equation}
q_{\lambda}=q_{0}+\lambda w.\label{eq:q.lambda}%
\end{equation}
We show in Appendix \ref{app:radial} that, for a radial potential, $\mathbf{t}(k)$ is a real-valued and radial function.  Our main result is:

\begin{theorem}
\label{thm:main}
Denote by $\mathbf{t}_{\lambda}$ the scattering transform of
$q_{\lambda}$. For small $\lambda \neq 0$,
\[
\mathbf{t}(k) = -\frac{2\pi}{\log |k|}+\mathcal{O}(k)
\]
as $k \rightarrow 0$. Moreover:
\newline(1) For $\lambda>0$ sufficiently small, the exceptional
set $\mathcal{E}$ is empty, and the scattering transform $\mathbf{t}_{\lambda
}$ is $\mathcal{C}^{\infty}$ away from $k=0$.  
\newline(2)\ For $\lambda<0$
sufficiently small and a unique $r(\lambda)>0$, the exceptional set
$\mathcal{E}$ is a circle $C_{\lambda}$ of radius $r(\lambda)$ about the
origin, and the function $\mathbf{t}_{\lambda}$ is $\mathcal{C}^{\infty}$ on
$\mathbb{R}^{2}\backslash\left[  C_{\lambda}\cup\left\{  0\right\}  \right]
$, while
\[
\lim_{\left\vert k\right\vert \rightarrow r\left(  \lambda\right)  }\left\vert
\mathbf{t}_{\lambda}(k)\right\vert =\infty\text{.}%
\]
The radius $r(\lambda)$ obeys the formula%
\begin{equation}
r(\lambda)\underset{\lambda\uparrow0}{\sim}\exp\left[  2\pi\left(
h+\frac{\left(  1+\mathcal{O}\left(  \lambda\right)  \right)  }{2\pi
\mu(\lambda)}\right)  \right]  \label{eq:r.lambda}%
\end{equation}
where
\[
h=-\frac{\gamma}{2\pi},
\]
$\gamma$ is Euler's constant, and $\mu(\lambda)$ is the eigenvalue of the
Dirichlet-to-Neumann operator for $q_{\lambda}$ corresponding to the constant
functions on $S^{1}$.
\end{theorem}


Theorem \ref{thm:main} shows that $\lambda=0$ is an `essential singularity'
for the map $\lambda\mapsto\mathbf{t}_{\lambda}$. The cases $\lambda=0$,
$\lambda<0$, and $\lambda>0$ may be characterized in the following way. We
recall from Murata \cite{Murata:1986} that a Schr\"{o}dinger operator
$-\Delta+q$ is called
\begin{itemize}
\item[(i)] \emph{subcritical} if $-\Delta+q$ has a positive Green's function,

\item[(ii)] \emph{critical} if $-\Delta+q$ does not have a positive Green's function,
but the quadratic form%
\begin{equation*}
\mathfrak{q}(v,v)=\int_{\mathbb{R}^2}\left( \left\vert (\nabla v)(z)\right\vert
^{2}+q(z)\left\vert v(z)\right\vert ^{2}\right)dz
\end{equation*}
on $\mathcal{C}_{0}^{\infty}(\mathbb{R}^{2})\times\mathcal{C}_{0}^{\infty
}(\mathbb{R}^{2})$ is nonnegative, and

\item[(iii)] \emph{supercritical} if the quadratic form $\mathfrak{q}$ is not nonnegative.
\end{itemize}

Appendix \ref{appendix:Murata} below shows how \cite{Murata:1986} implies the following. First of all, the conductivity-type potential $q_0$ is critical. Furthermore, taking $\lambda<0$ in (\ref{eq:q.lambda}) gives a supercritical potential $q_\lambda$ which cannot be of conductivity type since there is no positive solution of $\left(  -\Delta+q\right)  \psi=0$. Finally, taking $\lambda>0$ in (\ref{eq:q.lambda}) gives a subcritical potential $q_\lambda$ allowing a unique positive solution $\psi$ of $\left(  -\Delta+q\right)  \psi=0$. However, this $\psi(z)$ grows logarithmically in $|z|$ and does not satisfy $\lim_{\left\vert z\right\vert \rightarrow \infty}\psi(z)=1$, so $q_\lambda$ is not of conductivity type in the sense of Definition \ref{def:condtype}. Therefore, the class of conductivity-type potentials is not stable under perturbations of $\lambda$.

The scattering transform $\mathcal{T}:q\rightarrow\mathbf{t}$ plays an
important role not only in Nachman's solution of the inverse conductivity
problem, but also in the solution of the Novikov-Veselov (NV) equation by the
method of inverse scattering (see Lassas, Mueller, Siltanen and Stahel
\cite{LMSS:2011} and Perry \cite{Perry:2012} for details and further
references). Thus, a clear understanding of the singularities of $\mathcal{T}$
is important both for Schr\"{o}dinger inverse scattering and for understanding
the dynamics of the NV\ equation and other completely integrable equations in
the NV\ hierarchy.

It follows from Perry \cite{Perry:2012} that the Novikov-Veselov equation with initial data $q_0$ has a global solution. On the other hand, Taimanov and Tsarev \cite{TT:2007,TT:2008a,TT:2008b,TT:2010} have used the Moutard transformation to construct initial data $q$ for the Novikov-Veselov so that the solutions blow up in finite time. These potentials have $L^2$ eigenvalues at zero energy, and hence are not of conductivity type (see Proposition \ref{prop:cond.spec}). We conjecture the following dichotomy for the Cauchy
problem for the Novikov-Veselov equation:
\begin{itemize}
\item[(a)] If the initial data is critical or slightly subcritical, the solution exists globally in
time, and
\item[(b)] In all other cases, the solution blows up in finite time.
\end{itemize}
We hope to return to this question in a subsequent paper.

We close this introduction by sketching the proof of Theorem \ref{thm:main}
and summarizing the contents of this paper. To analyze the singularities of
the scattering transform $\mathbf{t}_{\lambda}$ for the family
(\ref{eq:q.lambda}), we recall the reduction of (i) the problem (\ref{eq:CGO})
and (ii) the map (\ref{eq:t}) respectively to (i) a boundary integral equation
of Fredholm type (see (\ref{eq:CGO.b})), and (ii) a boundary integral (see
\ref{eq:t.b}). We refer the reader to Nachman \cite{Nachman:1996}, Theorem 5
and its proof in section 7 for a complete discussion.

In order to state the reduction, we first define the Dirichlet-to-Neumann map,
denoted $\Lambda_{q}$, for the Dirichlet problem%
\begin{align}
\left(  -\Delta+q\right)  u  &  =0\text{ in }B_{1}\label{eq:DP}\\
\left.  u\right\vert _{S^{1}}  &  =f.\nonumber
\end{align}
If zero is not an eigenvalue of the operator $-\Delta+q$ with Dirichlet boundary conditions on $B_1$,
the problem (\ref{eq:DP}) has a unique solution $u$
 for given $f\in
H^{1/2}(S^{1})$, and we set%
\begin{equation}\label{def:Lambda_q}
\Lambda_{q}f=\left.  \frac{\partial u}{\partial\nu}\right\vert _{S^{1}}%
\end{equation}
where $\partial/\partial\nu$ denotes differentiation with respect to the
outward normal on $S^{1}$. We let $\langle \cdot,
\Lambda_q\cdot \rangle$ denote the corresponding bilinear form:
\begin{equation} \label{eq:BF}
 \langle g, \Lambda_q f \rangle = \int_D (\nabla v \cdot \nabla u + qvu)dz \qquad \left(=\int_{S^1} v \frac{\partial u}{\partial \nu}dS\right),
\end{equation}
where $u$ solves (\ref{eq:DP}) and $v\in H^1(D)$ with $\left. v\right\vert_{S^1}=g$. We will denote by $\Lambda_{0}$ the
Dirichlet-to-Neumann operator for (\ref{eq:DP}) with $q=0$. It is known that
$\Lambda_{q}:H^{1/2}(\partial\Omega)\rightarrow H^{-1/2}(\partial\Omega)$. If
$\psi$ denotes the restriction of the unique solution of (\ref{eq:CGO}) to
$S^{1}$, then%
\begin{equation}
\left.  \psi\right\vert _{S^{1}}=e^{ikz}-T\psi\label{eq:CGO.b}%
\end{equation}
where $T:H^{1/2}(S^{1})\rightarrow H^{1/2}(S^{1})$ is the compact operator%
\begin{equation}
T\psi=\mathcal{S}_{k}\left(  \Lambda_{q}-\Lambda_{0}\right)  \psi.
\label{eq:T}%
\end{equation}
Here, the operator $\mathcal{S}_{k}$ is an integral operator%
\begin{equation}
\left(  \mathcal{S}_{k}\psi\right)  (z)=\int_{S^{1}}G_{k}(z-z^\prime)\psi
(z^\prime)~dS(z^\prime), \label{eq:Sk}%
\end{equation}
$G_{k}(~\cdot~)$ is Faddeev Green's function (see (\ref{eq:Gk}) in what
follows), and $dS$ is arc length measure on $S^{1}$. The formula%
\begin{equation}
\mathbf{t}(k)=\int_{S^{1}}e^{i\overline{k}\overline{z}}\left[  \left(
\Lambda_{q}-\Lambda_{0}\right)  \psi\right]  (z,k)~dS(z) \label{eq:t.b}%
\end{equation}
computes the scattering transform in terms of the boundary data $\psi$ that
solve (\ref{eq:CGO.b}).

To prove Theorem \ref{thm:main}, we will study the family of operators
$T_{k,\lambda}$ corresponding to (\ref{eq:T}) with $q$ given by
(\ref{eq:q.lambda}). We will show that the resolvent $\left(  I+T_{k,\lambda}\right)  ^{-1}$ has a rank-one singularity if $\lambda<0$
 for $\left\vert k\right\vert =r(\lambda)$, where the asymptotics of $r\left(
\lambda\right)  $ are given by (\ref{eq:r.lambda}). We will then use
(\ref{eq:t.b})\ to show that this rank-one singularity leads to a singularity
in $\mathbf{t}\left(  k\right)  $ on the circle of radius $r(\lambda)$.

The structure of this paper is as follows. In \S \ref{sec:prelim} we recall
some basic facts needed to analyze (\ref{eq:CGO.b}) and the associated
operators. In \S \ref{sec:singular} we extract the rank-one singularity of the
resolvent $\left(  I+T\right)  ^{-1}$ and show the existence of a `circle of
singularities' for $\mathbf{t}_{\lambda}$. Finally, in \S \ref{sec:numerics},
we present numerical computations of $\mathbf{t}_{\lambda}$ and compare the results to the analysis in
\S \ref{sec:singular}.

\textbf{Acknowledgements}. Two of us (P.P. and S.S.) thank the Isaac Newton
Institute for hospitality during part of the time this work was carried out. The work of S.S. was supported in part by the Academy of Finland (Finnish Centre of Excellence in Inverse Problems Research 2006--2011 and 2012--2017, decision numbers 213476 and 250215). Funding from Teknologiateollisuus r.y. was used to cover Microsoft Azure cloud computing resources. Also, we thank Techila Ltd. for providing technical support and discounts for parallelization middleware solutions used in the computations.
M.M. and P.P. thank the National Science Foundation for support under grant DMS-1208778.

\section{Preliminaries}\label{sec:prelim}

\noindent
First, we recall that $H^{1/2}(S^{1})$ can be defined in terms of the Fourier
basis $\left\{  \varphi_{n}\right\}  _{n=-\infty}^{\infty}$ of $L^{2}(S^{1})$
given by%
\begin{equation}
\varphi_{n}(\theta)=\frac{e^{in\theta}}{\sqrt{2\pi}}\label{eq:phi.n}%
\end{equation}
in the following way. A\ function $f\in L^{2}(S^{1})$ belongs to
$H^{1/2}(S^{1})$ if \ $f=\sum_{n}a_{n}\varphi_{n}$ and
\begin{equation}
\left\Vert f\right\Vert _{H^{1/2}(S^{1})}^{2}=\sum_{n=-\infty}^{\infty}\left(
1+\left\vert n\right\vert \right)  \left\vert a_{n}\right\vert ^{2}%
\label{eq:H1/2}%
\end{equation}
is finite. If we denote by $P$ the projection%
\begin{equation}
\left(  P\psi\right)  (z)=\frac{1}{2\pi}\int_{S^{1}}\psi(w)~dS(w),
\label{eq:Pprojection}%
\end{equation}
and let $Q=I-P$, then $P$ and $Q$ are orthogonal projections in $H^{1/2}%
(S^{1})$. We denote by $H^{-1/2}(S^{1})$ the topological dual of
$H^{1/2}(S^{1})$ and by $(~\cdot~,~\cdot~)$ the inner product associated to
the norm (\ref{eq:H1/2}).

Next, we recall some basic facts about the operators $\mathcal{S}_{k}$ and
$\Lambda_{q}$ that appear in the boundary integral equation (\ref{eq:CGO.b}).

First of all, we need some properties of the integral operator $\mathcal{S}%
_{k}$. A reference for this material is the thesis of Siltanen
\cite{Siltanen:1999}. For $k\in\mathbb{C}$, the operator $\mathcal{S}_{k}$ in
(\ref{eq:Sk}) is defined in terms of \emph{Faddeev's Green's function} (setting $z=x+iy$)
\begin{equation}
G_{k}(z)=e^{ikz}g_{k}(z)\label{eq:Gk}%
\end{equation}
where $kz$ is given by (\ref{eq:kx}) and%
\begin{equation}
g_{k}(z)=\frac{1}{(2\pi)^{2}}\int_{\mathbb{R}^{2}}\frac{e^{i\xi\cdot z}}%
{\xi\left(  \overline{\xi}+2k\right)  }d\xi_1d\xi_2.\label{eq:gk}%
\end{equation}
Here $\xi\cdot z = \xi_1 x+\xi_2 y$. In the denominator of (\ref{eq:gk}), $\xi=\xi_{1}+i\xi_{2}$. The function
$g_{k}$ is the convolution kernel for Green's function for the operator $-\frac{1}{4}\left(
\overline{\partial}\left(  \partial+k\right)  \right)  ^{-1}$. The factor
$e^{ikz}$ normalizes $G_{k}$ to be a fundamental solution for $\Delta
=4\partial\overline{\partial}$. Thus%
\begin{equation}
H_{k}(z)=G_{k}(z)-G_{0}(z)\label{eq:Hk}%
\end{equation}
is smooth and harmonic, where%
\[
G_{0}(z)=-\frac{1}{2\pi}\log\left\vert z\right\vert
\]
is the normalized fundamental solution for $-\Delta$. The integral operator
$\mathcal{S}_{k}$ is a bounded operator from $H^{-1/2}(S^{1})$ to
$H^{1/2}(S^{1})$ (see, for example, Lemma 7.1 of Nachman \cite{Nachman:1996}),
but we will need a finer description. The following Lemma is a simple
consequence of \cite{Siltanen:1999}, Theorem 3.2 and following discussion),
and we omit the proof.

\begin{lemma}
The formula%
\begin{equation}\label{H1repr}
H_{k}(z)=H_{1}(kz)-\frac{1}{2\pi}\log\left\vert k\right\vert ,
\end{equation}
holds. Here $H_{1}(~\cdot~)$ is real-valued, smooth, and harmonic, and
\[
H_{1}(0)=-\frac{\gamma}{2\pi},
\]
where $\gamma$ is Euler's constant.
\end{lemma}

From the Lemma, we immediately conclude:

\begin{lemma}
\label{lemma:Sk}The decomposition%
\begin{equation}
\mathcal{S}_{k}=\mathcal{S}_{0}+\mathcal{H}_{k}-\left(  \log\left\vert
k\right\vert \right)  P \label{eq:Sk.split}%
\end{equation}
holds, where
\begin{equation}
\left(  \mathcal{H}_{k}\psi\right)  (z)=\int_{S^{1}}H_{1}(k(z-z^\prime))\psi
(z^\prime)~dS(z^\prime)\text{.} \label{eq:Hk.bis}%
\end{equation}

\end{lemma}

\begin{remark}
\label{rem:Sk}By straightforward computation, $\mathcal{S}_{0}P=0$ so that%
\begin{equation}
\left(  \mathcal{S}_{k}P\psi\right)  (z)=\left[  2\pi H_{1}(kz)-\log\left\vert
k\right\vert \right]  \left(  P\psi\right)  (z) \label{eq:SkP}%
\end{equation}
where we have used the mean value property for harmonic functions, while
\begin{equation}
\mathcal{S}_{k}Q=\mathcal{S}_{0}Q+\mathcal{H}_{k}Q. \label{eq:SkQ}%
\end{equation}
We will use this decomposition to analyze the singularities of the operator
$T_{k,\lambda}$ as $\lambda\uparrow0$.
\end{remark}

Next, we need some simple properties of the Dirichlet-to-Neumann map and the
integral operator $T$ in the presence of radial symmetry. First, if $q$ is
radial, the problem (\ref{eq:DP}) can be solved by Fourier analysis on the
circle using the basis (\ref{eq:phi.n}). We set, for $z=re^{i\theta}$,%
\[
\psi(z,k)=\sum_{n=-\infty}^{\infty}\psi_{n}(r)\varphi_{n}(\theta).
\]
To compute the Dirichlet-to-Neumann map, we solve the problem
\begin{align}
(-\Delta +q) \psi =&0 \mbox{ in $B_1$} \notag\\
\left. \psi \right|_{S^1} =& \varphi_n \notag
\end{align}
Writing $\psi = \psi_n(r) \varphi_n(\theta)$ we have
\begin{align}
\label{eq:Dnevp}
-\frac{1}{r}\frac{\partial}{\partial r}\left(  r\frac{\partial\psi_{n}%
}{\partial r}\right)  +\left(  \frac{n^{2}}{r^{2}}+q(r)\right)  \psi_{n}
= &0\\
\psi_{n}(1)   = &1 \notag
\end{align}
It follows that
\begin{equation}
\Lambda_{q}\varphi_{n}=\mu_{n}(q)\varphi_{n}%
\label{eq:Lqevals}
\end{equation}
where%
\[
\mu_{n}(q)=\psi_{n}^{\prime}(1).
\]
Thus $\Lambda_{q}$ has a complete set of orthonormal eigenfunctions and real
eigenvalues, and hence commutes with complex conjugation. Furthermore, it is easy to see that
\begin{equation}\label{eigiden2}
  \mu_{-n} = \mu_n\in \mathbb{R}.
\end{equation}

Let
\begin{equation}\label{mulambdadef}
\mu(\lambda):=\mu_{0}(q_{\lambda}),
\end{equation}
where $q_{\lambda}$ is given by (\ref{eq:q.lambda}). The following fact is crucial.

\begin{lemma}
\label{lemma:mu}Suppose that $\sigma\in \mathcal{C}^\infty(\mathbb{R}^2)$ is a real-valued radial function satisfying $\sigma(z)>0$ for all $z\in\mathbb{R}^2$ and $\sigma-1\in\mathcal{C}_{0}^{\infty}\left(  B_{1}\right)$. Denote $q_{0}:=\sigma^{-1/2}\Delta\sigma^{1/2}\in
\mathcal{C}_{0}^{\infty}\left(  B_{1}\right)$. Let $w\in\mathcal{C}_{0}^{\infty}(B_{1})$ be a nonzero, nonnegative radial function, and define $q_{\lambda}$ by (\ref{eq:q.lambda}). Then $\mu(0)=0$ and $\mu^{\prime}(0)>0.$
\end{lemma}

\begin{proof}
The function $\sigma^{1/2}$ is the unique solution to $\left(  -\Delta+q_{0}\right) \phi = 0$ with $\left. \phi\right\vert _{S^{1}} =1$, and since $\sigma^{1/2}$ is constant in a neighborhood of the boundary we have $\left.\frac{\partial }{\partial\nu}\sigma^{1/2}\right\vert_{S^1}=0$.  It follows that $\mu(0)=0$. The fact that $\mu(\lambda)$ is continuously differentiable in $\lambda$ follows from the fact that the unique solution to the problem (\ref{eq:Dnevp}) with $q$ given by (\ref{eq:q.lambda}) depends analytically on $\lambda$.

To compute $\mu^{\prime}(0)$, let $\psi_{\lambda}$ solve the Dirichlet problem%
\begin{align}
\left(  -\Delta+q_{\lambda}\right)  \psi_\lambda &  =0,\label{eq:DP.lambda}\\
\left.  \psi_\lambda\right\vert _{S^{1}}  &  = 1,\nonumber
\end{align}
and let $\dot{\psi}_{\lambda}=\partial\psi_{\lambda}/\partial\lambda$. Since $\mu(\lambda) = \langle 1,\Lambda_{q_\lambda}1\rangle$, we calculate $\mu'(0)$ by taking the derivative of this pairing:
\begin{align*}
    \mu'(\lambda) &= \frac{d}{d\lambda} \langle 1,\Lambda_{q_\lambda} 1\rangle = \frac{d}{d\lambda}\int_D |\nabla \psi_\lambda|^2+q_\lambda|\psi_\lambda|^2 dz \\
    &=  \int_D 2\nabla \dot{\psi}_\lambda\cdot\nabla \psi_\lambda+2 q_\lambda \dot{\psi}_\lambda \psi_\lambda+ w |\psi_\lambda|^2 dz\\
    &= 2 \left\langle  \right.\dot{\psi}_\lambda\left\vert_{S^1},\Lambda_{q_\lambda} 1\right\rangle + \int_D w |\psi_\lambda|^2 dz.
\end{align*}

We have $\left.\dot \psi_\lambda\right\vert_{S^1} = 0$, so $\mu^{\prime}(0) = \int_D w|\psi_0|^2 dz>0$.

\end{proof}

\section{Singularities of CGO\ Solutions and Scattering Transform} \label{sec:singular}

\noindent
In this section we compute the singularities of solutions to the boundary
integral equation (\ref{eq:CGO.b}). Our computation is based on the following
decomposition of the resolvent $\left(  I+T_{k,\lambda}\right)  ^{-1}$ as a
bounded operator from $H^{1/2}(S^{1})$ to itself. We note that, since $q_{0}$
is of conductivity type, it follows from Theorem 5 in \cite{Nachman:1996} that
the operator $\left(  I+T_{k,0}\right)  $ has a bounded inverse for all
$k\in\mathbb{C}$.

\begin{lemma}
\label{lemma:FR}Suppose that $\sigma\in \mathcal{C}^\infty(\mathbb{R}^2)$ is a real-valued radial function satisfying $\sigma(z)>0$ for all
 $z\in\mathbb{R}^2$ and $\sigma-1\in\mathcal{C}_{0}^{\infty}\left(  B_{1}\right)$. Denote $q_{0}:=\sigma^{-1/2}\Delta\sigma^{1/2}\in
\mathcal{C}_{0}^{\infty}\left(  B_{1}\right)$. Let $w\in\mathcal{C}_{0}^{\infty}(B_{1})$ be a nonnegative and nonzero function, and define $q_{\lambda}$ by (\ref{eq:q.lambda}). Suppose that $T_{k,\lambda}$ is given by (\ref{eq:T}) with
$q=q_{\lambda}$. Then, for all $|\lambda|$ sufficiently small,
there is a rank-one operator $F=F(k,\lambda)$ and a bounded operator
$R=R(k,\lambda)$ so that
\begin{equation}
\left(  I+T_{k,\lambda}\right)  ^{-1}=\left(  I+T_{k,0}+R\right)  ^{-1}\left(
I+\mu F\right)  ^{-1}\label{eq:Tklambda.res}%
\end{equation}
where $\mu=\mu(\lambda)$ is defined in (\ref{mulambdadef}),
\begin{equation}
\sup_{\left\vert k\right\vert \leq K_0}\left\Vert R\right\Vert \leq C(K_0)\left\vert
\lambda\right\vert \label{eq:R}%
\end{equation}
for a positive constant $C$ depending on $q_{0}$ and $w$,
and
\[
F=\mathcal{S}_{k}P\left(  I+T_{k,0}+R\right)  ^{-1}.
\]
The operators $R$ and $F$ commute with complex conjugation.
\end{lemma}

\begin{proof}
Let%
\[
R=\mathcal{S}_{k}\left(  \Lambda_{q_{\lambda}}-\Lambda_{q_{0}}\right)  Q.
\]
Since $Q$ commutes with $\Lambda_{q_{\lambda}}$ and $\Lambda_{0}$ it follows
that
\[
R=\left(  \mathcal{S}_{k}Q\right)  \left(  \Lambda_{q_{\lambda}}%
-\Lambda_{q_{0}}\right)  Q.
\]
From Remark \ref{rem:Sk}, we have
\[
\mathcal{S}_{k}Q=\mathcal{S}_{0}Q+\mathcal{H}_{k}Q
\]
so that $\sup_{\left\vert k\right\vert \leq K_0}\left\Vert \mathcal{S}%
_{k}Q\right\Vert _{H^{-1/2}(S^{1})\rightarrow H^{1/2}(S^{1})}$ is finite. We
then use the fact that%
\[
\left\Vert \Lambda_{q_{\lambda}}-\Lambda_{q_{0}}\right\Vert _{H^{1/2}%
(S^{1})\rightarrow H^{-1/2}(S^{1})}\leq C\left\vert \lambda\right\vert
\]
to conclude that (\ref{eq:R})\ holds. Write%
\[
T_{k,\lambda}=\mu\mathcal{S}_{k}P+\mathcal{S}_{k}Q\left(  \Lambda_{q_{\lambda
}}-\Lambda_{q_{0}}\right)  +\mathcal{S}_{k}Q\left(  \Lambda_{q_{0}}%
-\Lambda_{0}\right)
\]
and use the fact that $T_{k,0}=\mathcal{S}_{k}Q\left(  \Lambda_{q_{0}}%
-\Lambda_{0}\right)  $ (since $\Lambda_{q_{0}}P=\Lambda_{0}P=0$) to conclude
that (\ref{eq:Tklambda.res}) holds. The statement about complex conjugation
follows from the fact that the projections $P$ and $Q$, the operators
$\Lambda_{q_{\lambda}}$, and the operator $\mathcal{S}_{k}$ have the same property.
\end{proof}

\begin{remark}
The fact that $T_{k,0}=\mathcal{S}_{k}Q\left(  \Lambda_{q_{0}}-\Lambda
_{0}\right)  $ implies that $\sup_{\left\vert k\right\vert \leq K_0}\left\Vert
T_{k,0}\right\Vert $ is bounded, and hence that $\sup_{\left\vert k\right\vert
\leq K_0}\left\Vert \left(  I+T_{k,0}\right)  ^{-1}\right\Vert $ is bounded since
$k\mapsto T_{k,0}$ is a bounded continuous operator-valued function for
$\left\vert k\right\vert \leq K_0$ and $\left(  I+T_{k,0}\right)  ^{-1}$ is known
to exist for all $k$ by \cite{Nachman:1996}.
\label{rm:bdd.inv}
\end{remark}

The effect of Lemma \ref{lemma:FR} is to focus attention on the rank-one
operator $F$. If we write%
\[
F\psi=(\varphi,\psi)\chi
\]
a short computation shows that%
\[
\left(  I+\mu F\right)  ^{-1}=I-\mu\frac{(\varphi,~\cdot~)\chi}{1+\mu\left(
\varphi,\chi\right)  }%
\]
while%
\[
D(k,\lambda):=\det\left(  I+\mu F\right)  =1+\mu\left(  \varphi,\chi\right)  .
\]
In our case we have, with $\varphi_{0}$ defined in (\ref{eq:phi.n}),%
\begin{align}
\varphi &  =\left(  I+T_{k,0}^{\ast}+R^{\ast}\right)  ^{-1}\varphi
_{0},\label{eq:phi1}\\
\chi &  =\left[  H_{1}(k~\cdot~)+\log\left(  \left\vert k\right\vert \right)
\right]  \varphi_{0}.\label{eq:psi1}%
\end{align}
We can prove:

\begin{lemma}
For any $K_0>0$ and all $|k|\leq K_0$ and $\lambda$ sufficiently small, the
function $D(k,\lambda)$ is real-valued, smooth in $\lambda$ and $k$
for  $k\neq0$, and radial in $k$.
Moreover, $D(k,\lambda)$ obeys the small-$|\lambda|$ asymptotics%
\begin{equation}
D(k,\lambda)=1+\mu(\lambda)\left(  2\pi h-\log\left\vert k\right\vert \right)
+\mathcal{O}\left(  \left\vert \lambda\right\vert \left\vert k\right\vert
\right)  . \label{eq:D.asy}%
\end{equation}

\end{lemma}

\begin{proof}
From (\ref{eq:phi1})-(\ref{eq:psi1}) we compute $\left(  \varphi,\chi\right)
=F_{1}+F_{2}$, where%
\begin{align}
F_{1}(k,\lambda)  &  =(2\pi h-\log\left\vert k\right\vert )\left(  \varphi
_{0},\left(  I+T_{k,0}+R\right)  ^{-1}\varphi_{0}\right)  ,\label{eq:overlap1}%
\\
F_{2}(k,\lambda)  &  =\left(  \left(  I+T_{k,0}^{\ast}+R^{\ast}\right)
^{-1}\varphi_{0},H_{1}(k~\cdot~)\varphi_{0}\right)  . \label{eq:overlap2}%
\end{align}
In (\ref{eq:overlap1}), $h=H_{1}(0)=-\gamma/2\pi$, and, in (\ref{eq:overlap2}%
), $H(z)=H_{1}(z)-H_{1}(0)$. Since%
\[
\left(  I+T_{k,0}+R\right)  ^{-1}\varphi_{0}=\varphi_{0},
\]
we have%
\[
F_{1}(k,\lambda)=\mu(\lambda)\left(  2\pi h-\log\left\vert k\right\vert
\right)
\]
which is obviously radial in $k$. To see that $F_{2}(k,\lambda)$ is radial in
$k$, we note that if%
\[
C_{k}=T_{k,0}+R
\]
and $U_{\phi}$ is the unitary operator on $H^{1/2}(S^{1})$ given by $\left(
U_{\phi}f\right)  \left(  e^{i\theta}\right)  =f(e^{i(\theta+\phi)})$, we have%
\[
U_{\phi}C_{k}=C_{e^{i\phi}k}U_{\phi}%
\]
and%
\[
U_{\phi}\mathcal{H}_{k}=\mathcal{H}_{e^{i\phi}k}U_{\phi}.
\]
The second identity implies that $H_{1}(kz)\varphi_{0}=\mathcal{H}_{k}%
\varphi_{0}$ satisfies $U_{\phi}\left(  H_{1}(k~\diamond\right)  \varphi
_{0}\left(  \diamond\right)  )=H_{1}(e^{i\phi}k~\diamond)\varphi_{0}$. Using
these identities and the fact that $U_{\phi}\varphi_{0}=\varphi_{0}$ we easily
deduce that $F_{2}(e^{i\phi}k,\lambda)=F_{2}(k,\lambda)$ so that $F_{2}$ is
radial in $k$. \ Finally,
\[
\left\vert F_{2}(k,\lambda)\right\vert \leq C\left\Vert \left(  I+T_{k,0}%
^{\ast}+R^{\ast}\right)  ^{-1}\right\Vert \sup_{\left\vert z\right\vert \leq
 1}\left\vert H\left(  k~\cdot~\right)  \right\vert \leq C\left\vert
k\right\vert ,
\]
which gives the required asymptotics since $\left\vert \mu(\lambda)\right\vert
\leq C\left\vert \lambda\right\vert $ for small $|\lambda|$. The reality of
$D\left(  k,\lambda \right)  $ follows from the fact that $\varphi_{0}$ and
$H(k~\cdot~)$ are real-valued functions and that the operators $T_{k,0}$ and
$R$ commute with complex conjugation. Smooth dependence on the parameters $k$ and $\lambda$ is a consequence of the smooth dependence of $T_{k,\lambda}$ on these parameters.
\end{proof}

\begin{remark}\label{remark:circle}
It follows from the lemma that for $|k|\leq K_0$ and $\lambda$ sufficiently close to zero and
$k\neq0$,
\[
\left\vert \nabla_{k}D(k,\lambda)\right\vert =\frac{\left\vert\mu(\lambda)\right\vert}{\left\vert
k\right\vert }\left(  1+\mathcal{O}(\lambda)\right)
\]
is nonzero, so that the zero set
\[
Z_{\lambda}=\left\{  k\in\mathbb{R}^{2}:k\neq0,~D(k,\lambda)=0\right\}
\]
is locally a smooth curve. By radial symmetry, $Z_{\lambda}$ is actually a
circle of radius $r_{\lambda}$ depending on $\lambda$. We easily deduce from
(\ref{eq:D.asy})\ and Lemma \ref{lemma:mu} that
\begin{equation}
r_{\lambda}\sim\exp\left[  2\pi\left(  h+\frac{1+\mathcal{O}\left(
\lambda\right)  }{2\pi\mu^{\prime}(0)\lambda}\right)  \right].
\label{eq:r.asy}%
\end{equation}
\end{remark}
\begin{corollary}
For $\lambda>0$ small, the exceptional set $\mathcal E$ is empty.
 \label{cor:lam.pos}
\end{corollary}
\begin{proof}
We see from (\ref{eq:r.asy}) that for fixed $K_0$ and small positive $\lambda$, $r_\lambda>K_0$
and thus there are no singularities for $|k|<K_0$. From Theorem 2.1 of \cite{S-U:1986},
 if $|k|\geq C\|(1+|z|^2)^{1/2}q_\lambda\|_{L^{\infty}}$ then the modified CGO solutions $e^{-ikz}\psi(z,k)$,
exist, are unique, and are locally integrable. The function $q_\lambda$ has support in
$|z|<1$, therefore choosing $K_0 = \sup_{0\leq\lambda\leq1} C\|2q_\lambda\|_{L^{\infty}}$
and using (\ref{eq:t}), there are no singularities for $|k|\geq K_0$.
\end{proof}

The purpose of the following lemma is to show that, even when there are no exceptional points of $\mathbf{t}_\lambda(k)$, this does not violate Nachman's result \cite[Theorem 3]{Nachman:1996} because $\mathbf{t}_\lambda(k)$ does not satisfy the small $k$ decay requirement.

\begin{lemma}
\label{lemma.0}
If $(I+T_{k,\lambda})^{-1}$ exists and is bounded for small $k\neq 0$ and $\mu(\lambda) \neq 0$ then
\[\mathbf{t}_\lambda(k) = -\frac{2\pi}{\log\left\vert k \right \vert }+\mathcal O(k)\]
\end{lemma}

\begin{proof}
We expand $e^{ikz}$ and $e^{i\bar k \bar z}$ for small $k$ as $1+\mathcal O(|k|)$ and then calculate  $(I+T_{k,\lambda})^{-1}$ explicitly.
\begin{align}
 \mathbf{t}_\lambda(k) &= \int_{S^1} e^{i\bar k \bar z} (\Lambda_{q_\lambda}-\Lambda_0)(I+T_{k,\lambda})^{-1} e^{ikz} \,dS \notag\\
&= \int_{S^1} (1+\mathcal O (k)) (\Lambda_{q_\lambda}-\Lambda_0)(I+T_{k,\lambda})^{-1}(1+\mathcal O (k))\,dS \notag\\
&= \int_{S^1} (\Lambda_{q_\lambda}-\Lambda_0)(I+T_{k,\lambda})^{-1}1\,dS + \mathcal O(k) \label{eq:tksmall}
\end{align}
Expanding $(I+T_k)1$
 using (\ref{eq:Sk.split}) gives us
\begin{align*}
(I+T_{k_\lambda})1&=[I+(\mathcal{S}_{0}+\mathcal{H}_{k}-\left(  \log\left\vert
k\right\vert \right)  P)(\Lambda_{q_\lambda}-\Lambda_0)]1\\
 &=1+ \mu(\lambda)(2\pi h -\log\left\vert k\right\vert),
\end{align*}
and so $(I+T_{k,\lambda})^{-1}1 = 1/[1+\mu(\lambda)(2\pi h-\log\left\vert k \right\vert)]$. Applying this to (\ref{eq:tksmall}) gives
\[\mathbf{t}_\lambda(k) = \frac{2\pi \mu(\lambda)}{1+\mu(\lambda)(2\pi h - \log\left \vert k \right \vert)}+\mathcal O (k) = -\frac{2\pi}{\log\left\vert k \right\vert}+\mathcal O (k).\]

\end{proof}
For small $\lambda\neq 0$, we see the assumptions of the lemma are satisfied. Remark (\ref{rm:bdd.inv}) gives us that $(I+T_{k,\lambda})^{-1}$ is bounded, and lemma (\ref{lemma:mu}) gives us that $\mu(\lambda)\neq 0$. Therefore we have that \[\left\vert \mathbf{t}_\lambda(k)\right \vert = \left\vert\frac{-2\pi}{\log\left\vert k \right\vert}+\mathcal O(k)\right\vert> c\left\vert k \right \vert^{\epsilon}\]
for all $c,\epsilon>0$ and $k$ small, which shows that $\mathbf{t}_\lambda(k)$ does not decay fast enough for Nachman's Theorem 3 to apply.

Next, we show:

\begin{lemma}
There is a $\lambda_{0}>0$ so that, for all $\lambda$ with $0<-\lambda < \lambda_0$ and any $k_{c}\in Z_{\lambda}$
\[
\lim_{\substack{k\rightarrow k_{c}\\k\notin Z_{\lambda}}}\left\vert
\mathbf{t}_{\lambda}(k)\right\vert =\infty.
\]

\end{lemma}

\begin{proof}
From the formula%
\[
\psi(z,k)=\left(  I+T_{k,\lambda}\right)  ^{-1}\left(  e^{ik\left(
\diamond\right)  }\right)  (z)
\]
we compute%
\begin{align}
\label{eq:psiform}
\psi(z,k)  &  =\left[  \left(  I+T_{k,0}+R\right)  ^{-1}\left(  I+\mu
F\right)  ^{-1}\left(  e^{ik\left(  \diamond\right)  }\right)  \right]  (z)\\
 = & \left[  \left(  I+T_{k,0}+R\right)  ^{-1}\left(  I-\frac{\mu(\lambda
)}{D(k,\lambda)}\left(  \varphi,e^{ik\left(  \diamond\right)  }\right)
\chi(\diamond)\right)  \right]  (z) \notag \\
= &\frac{\mu(\lambda)\left(  \varphi,e^{ik\left(  \diamond\right)  }\right)
}{D(k,\lambda)}\left[  \left(  I+T_{k,0}+R\right)  ^{-1}\chi\right]
(z)+R(z,k) \notag
\end{align}
where $R(z,k)$ is regular in $z,k$. If $\lambda<0$ and $k_{0}\in
Z_{\lambda}$, the zero set of $D(k,\lambda)$, we have
\[
\lim_{k\rightarrow k_{c}}\left\vert \frac{1}{D(k,\lambda)}\right\vert
=\infty\text{.}%
\]
Formula (\ref{eq:t.b}) and the fact that $\mu(\lambda)$ is nonzero for small nonzero $\lambda$ imply that, to show
\[
\lim_{k\rightarrow k_{c}}\left\vert \mathbf{t}_{\lambda}(k)\right\vert
=\infty,
\]
it suffices to show that%
\begin{equation}
\underset{k\rightarrow k_{c}}{\lim\inf}\left\vert \left(  \varphi,e^{ik\left(
\diamond\right)  }\right)  \right\vert >0 \label{eq:liminf1}%
\end{equation}
and%
\begin{equation}
\underset{k\rightarrow k_{c}}{\lim\inf}\left\vert \int_{S^{1}}e^{i\overline
{k}\overline{z}}\left[  \left(  \Lambda_{q}-\Lambda_{0}\right)  \left(
I+T_{k,0}+R\right)  ^{-1}\chi\right]  (z)~dS(z)\right\vert >0
\label{eq:liminf2}%
\end{equation}
where $\varphi$ and $\chi$ are given respectively by (\ref{eq:phi1}) and
(\ref{eq:psi1}). Now%
\begin{align*}
\varphi &  =\left(  I+T_{k,0}^{\ast}+R\right)  ^{-1}\varphi_{0}\\
&  =\left(  I+T_{k,0}^{\ast}\right)  ^{-1}\varphi_{0}+\mathcal{O}\left(
\lambda\right)  .
\end{align*}
Writing%
\[
\left(  I+T_{k,0}^{\ast}\right)  ^{-1}\varphi_{0}=\varphi_{0}-\left(
I+T_{k,0}^{\ast}\right)  ^{-1}T_{k,0}^{\ast}\varphi_{0},
\]
using the fact that
\begin{align*}
T_{k,0}^{\ast}\varphi_{0}  &  =\left(  \Lambda_{q_{0}}-\Lambda_{0}\right)
Q\mathcal{S}_{k}\varphi_{0}\\
&  =\left(  \Lambda_{q_{0}}-\Lambda_{0}\right)  Q\left(  H(k~\cdot~) \right)
\end{align*}
(see (\ref{eq:Sk.split})) satisfies%
\[
\left\Vert T_{k,0}^{\ast}\varphi_{0}\right\Vert \leq C\left\vert k\right\vert
\text{,}%
\]
and using the fact that $\left\vert k_{c}\right\vert \sim Ce^{c/\lambda}$, we
conclude that, for $\lambda<0$ small, $\varphi=\varphi_{0}+\mathcal{O}\left(  \lambda\right)  $, that
$e^{ik_{c}z}\sim1+\mathcal{O}\left(  e^{c/\lambda}\right) $ ($c$ is a positive constant), and hence that
(\ref{eq:liminf1}) holds.

To prove (\ref{eq:liminf2}), first note that
\begin{align*}
\chi &  =\left(  H(k~\cdot~)+\log\left\vert k\right\vert \right)  \varphi
_{0}\\
&  \sim-\frac{1+\mathcal{O}\left(  \lambda\right)  }{\mu(\lambda)}\varphi
_{0}+\mathcal{O}\left(  e^{c/\lambda}\right)
\end{align*}
since $H$ is smooth and $\left\vert k\right\vert \sim e^{c/\lambda}$. Next,
note that $\left(  I+T_{k,0}+R\right)  ^{-1}\varphi_{0}=\varphi_{0}$ so that
finally%
\begin{multline*}
\left\vert \int_{S^{1}}e^{i\overline{k}\overline{z}}\left[  \left(
\Lambda_{q}-\Lambda_{0}\right)  \left(  I+T_{k,0}+R\right)  ^{-1}\chi\right]
(z)~dS(z)\right\vert \\
\geq\left\vert \int_{S^{1}}e^{i\overline{k}\overline{z}}\varphi_{0}%
~dS(z)\right\vert +\mathcal{O}\left(  e^{c/\lambda}\right)
\end{multline*}
which shows that (\ref{eq:liminf2})\ holds. This proves the lemma.
\end{proof}

\begin{proof}
[Proof of Theorem \ref{thm:main}]We have already shown that, for all
sufficiently small negative $\lambda$, the scattering transform $\mathbf{t}%
_{\lambda}$ is singular on a circle of radius $r_{\lambda}$ with the
asymptotic behavior (\ref{eq:r.asy}), and for small positive $\lambda$ the set $\mathcal E$ is empty by Corollary \ref{cor:lam.pos}.  The behavior of $\mathbf{t}_\lambda(k)$ near $k=0$ is given by Lemma \ref{lemma.0}.  It remains to show that $\mathbf{t}%
_{\lambda}$ is smooth elsewhere. This follows from the fact that $D(k,\lambda)$ is smooth and nonzero away from the singular circle, the formula (\ref{eq:psiform}), and the explicit formula for $\mathbf{t}_\lambda$.
\end{proof}

\section{Computational methods}  \label{sec:numerics}

\subsection{Evaluating eigenvalues of DN maps}\label{sec:numericalevals}

Given a radial potential $q$, we wish to compute numerically the eigenvalues $\mu_n(q)$ defined in (\ref{eq:Lqevals}). The straightforward approach would be this: use the finite element method (FEM) to solve the Dirichlet problem (\ref{eq:DP}) with $f=\varphi_n$. Then evaluate $\Lambda_q\varphi$ directly from (\ref{def:Lambda_q}) by numerical differentiation of the FEM solution. We will actually compute $\mu_0(q)$ in this way, but for $\mu_n(q)$ with $n\neq 0$ we can avoid the instability of numerical differentiation as explained below.

Consider the Neumann problem
\begin{equation}\label{NeumProb}
  (-\Delta + q)u=0 \mbox{ in }\Omega, \qquad \frac{\partial u}{\partial\nu}|_{\partial\Omega}=g,
\end{equation}
where the mean value of $g\in H^{-1/2}(\partial\Omega)$ is zero. Define the Neumann-to-Dirichlet map by
$$
 \mathcal{R}_q g:= u|_{\partial\Omega},
$$
where the condition $\int_{\partial\Omega} u|_{\partial\Omega} \,ds = 0$.
The functions $\varphi_n$ are eigenfunctions for the ND map: $\mathcal{R}_q\varphi_n = \nu_n(q)\varphi_n$.

Taking $f=\varphi_n$ with $n\not=0$ in (\ref{eq:DP}) results in
\begin{equation}\label{DirProbB}
  (-\Delta + q)u_n=0 \mbox{ in }\Omega, \qquad u_n|_{\partial\Omega}=\varphi_n,
\end{equation}
and we know that
$$
 \frac{\partial u_n}{\partial\nu}|_{\partial\Omega} = \Lambda_q\varphi_n=\mu_n(q)\varphi_n.
$$
Consider the Neumann problem
\begin{equation}\label{NeumProbB}
  (-\Delta + q)u_n=0 \mbox{ in }\Omega, \qquad \frac{\partial u_n}{\partial\nu}|_{\partial\Omega}=\mu_n(q)\varphi_n.
\end{equation}
Now since the solution $u_n$ is the same in (\ref{DirProbB}) and (\ref{NeumProbB}) we see that $\mathcal{R}_q(\mu_n(q)\varphi_n)=\varphi_n$. If $\mu_n(q)\not=0$, by linearity we get the following connection between the eigenvalues of the  {\sc dn}  and {\sc nd} maps:
\begin{equation}\label{Nevals}
  \nu_n(q)\varphi_n = \mathcal{R}_q\varphi_n = \frac{1}{\mu_n(q)}\varphi_n.
\end{equation}
Equation (\ref{Nevals}) provides us with a stable way to compute $\mu_n(q)$ using the finite element method for the solution of (\ref{NeumProb}), since there is no numerical differentiation involved in the evaluation of $\nu_n(q)$.

\subsection{Solution of the boundary integral equation}

\noindent
We explain how to solve equation (\ref{eq:CGO.b}) approximately by numerical computation. We follow the method described in \cite{Knudsen2009}. The trick is to write the integral equation approximately as a matrix equation on the truncated Fourier series domain.

Choose $N>0$. We represent a function $f\in H^s(\partial\Omega)$ approximately by the truncated Fourier series vector
$$
  \widehat{f}:=
\left[\!\begin{array}{l}
  \widehat{f}(-N)\\
  \widehat{f}(-N+1)\\
  \vdots\\
  \widehat{f}(0)\\
  \vdots\\
  \widehat{f}(N-1)\\
  \widehat{f}(N)
\end{array}\!\right],
$$
where the Fourier coefficients are defined for $-N\leq n \leq N$ by
$$
 \widehat{f}(n) := \langle f,\varphi_n\rangle =  \frac{1}{\sqrt{2\pi}}\int_0^{2\pi} f(\theta)\, e^{-in\theta}\, d\theta.
$$
By standard Fourier series theory we get for well-behaved $f$
$$
  f(\theta) \approx \sum_{n=-N}^N  \widehat{f}(n) \varphi_n(\theta).
$$

Our goal is to approximate the operator $T=\mathcal{S}_{k}\left(  \Lambda_{q}-\Lambda_{0}\right)$ using a matrix acting on the truncated Fourier basis.

We know analytically that $\Lambda_0\varphi_n=|n|\varphi_n$, so the $(2N+1)\times(2N+1)$ matrix representing the operator $\Lambda_0$ is
\begin{equation}\label{L0matrix}
  \mathbf{L}_0:= \mbox{diag}[N,N-1\dots,2,1,0,1,2,\dots,N-1,N].
\end{equation}
Likewise, the {\sc dn} map $\Lambda_q$ can be represented by a diagonal matrix containing the eigenvalues $\mu_{n}(q)$ defined in
(\ref{eq:Lqevals}):
\begin{equation}\label{Lqmatrix}
  \mathbf{L}_q:=\mbox{diag}[\mu_{N}(q),\mu_{N-1}(q),\dots,\mu_{1}(q),\mu_{0}(q),\mu_{1}(q),\dots,\mu_{N-1}(q),\mu_{N}(q)].
\end{equation}
Note that in (\ref{Lqmatrix}) we made use of the symmetry (\ref{eigiden2}).

By Lemma \ref{lemma:Sk} we can write $\mathcal{S}_{k}=\mathcal{S}_{0}+\mathcal{H}_{k}-\left(  \log\left\vert
k\right\vert \right)  P$.  In our case of $\Omega$ being the unit disc, the standard single layer operator $\mathcal{S}_{0}$ has the matrix
$$
  \mathbf{S}_0=\frac 12 \mbox{diag}[\frac{1}{N},\frac{1}{N-1}\dots,\frac 12,1,0,1,\frac 12,\dots,\frac{1}{N-1},\frac{1}{N}].
$$
Furthermore, the projection operator $P$ defined in (\ref{eq:Pprojection}) can be represented by
$$
  \mathbf{P}=\mbox{diag}[0,\dots,0,1,0,\dots,0].
$$

It remains to find a matrix $\mathbf{H}_k$ for the operator $\mathcal{H}_k$. We define the elements of $\mathbf{H}_k=[\mathbf{H}_k(m,n)]$
by
\begin{equation}\label{matrixeq}
   \mathbf{H}_k(m,n) := \langle \mathcal{H}_k\varphi_n,\varphi_m\rangle=\frac{1}{2\pi}\int_0^{2\pi} (\mathcal{H}_k e^{in\theta})\, e^{-im\theta}\, d\theta.
\end{equation}
Here $m\in\{-N,\dots,N\}$ is the row index and $n\in\{-N,\dots,N\}$ is the column index. The function $\mathcal{H}_k e^{in\theta}$ can be evaluated numerically by applying a quadrature rule to the integral in (\ref{eq:Hk.bis}). For this we need to be able to compute point values of $H_k(z)$. By (\ref{H1repr}) and (\ref{eq:Hk}) and (\ref{eq:Gk}) we can write
\begin{equation}  \label{Hkformula3}
  H_k(z) =
  H_1(kz)-\frac{\log|k|}{2\pi}
  =
  e^{ikz}g_1(kz)-G_0(kz)-\frac{\log|k|}{2\pi}.
\end{equation}
Now the evaluation of $H_k(z)$ is reduced to computing Faddeev's fundamental solution $g_1(z)$, since everything else is explicit in the right hand side of (\ref{Hkformula3}). Following \cite[(3.10)]{boiti}, that can be done simply using formula
\begin{equation}
  g_1(z) = \frac{1}{4\pi}e^{-iz}\mbox{Re}(\mbox{Ei}(iz)),
\end{equation}
where Ei stands for the exponential-integral special function whose implementation is readily available in mathematical software packages. As explained in \cite{Siltanen:1999}, one can avoid evaluating the functions $g_1(z)$ and $G_0(z)$ in (\ref{Hkformula3}) near the singularity at $z=0$ by calculating the harmonic function $H_k(z)$ first on a circle $|z|=R$ enclosing the evaluation domain, and then using the classical Poisson kernel to calculate $H_k(z)$ for $|z|<R$.

Approximate solution of (\ref{eq:CGO.b}) is now given in the frequency domain by
\begin{eqnarray}\label{def:psiBIE}
  \widehat{\psi|_{\Omega}}
  &=& \nonumber
  [I+\mathbf{S}_k(\mathbf{L}_q-\mathbf{L}_0)]^{-1}(\widehat{e^{ikz}|_{\partial\Omega}})\\
  &=& \label{psiFouriersol}
  [I+(\mathbf{S}_0+\mathbf{H}_k-(\log|k|)\mathbf{P})(\mathbf{L}_q-\mathbf{L}_0)]^{-1}(\widehat{e^{ikz}|_{\partial\Omega}}),
\end{eqnarray}
where we used the decomposition (\ref{eq:Sk.split}), and $\widehat{e^{ikz}|_{\partial\Omega}}$ stands for the Fourier expansion of $e^{ikz}$, calculated as follows. Write $z=e^{i\theta}$ and compute as in \cite[Section 2]{DM} to get
$$
 e^{ikz}=\sum_{n=-\infty}^{\infty}a_n(k)e^{in\theta} \quad \mbox{with} \quad
a_n(k)=
\left \{ \begin{array}{cl}
\frac{(ik)^n}{n!}, & n\geq 0, \\ \\
0, & n<0.
\end{array}\right.
$$
The vector $\widehat{e^{ikz}|_{\partial\Omega}}$ thus takes the explicit form
$$
  \widehat{e^{ikz}|_{\partial\Omega}}=
\sqrt{2\pi}
\left[\!\!\begin{array}{c}
  0\\
  0\\
  \vdots\\
  0\\
  1\\
  ik\\
  -k^2/2\\
  \vdots\\
  (ik)^N/N!
\end{array}\!\!\right],
$$

Recall now the infinite-precision formula
$$
  \mathbf{t}(k) = \int_{\partial\Omega} e^{i\overline{k}\overline{z}}(\Lambda_q-\Lambda_0)\psi(\,\cdot\,,k)\,ds.
$$
Once the Fourier coefficient vector $\widehat{\psi|_{\Omega}}$ is solved from  (\ref{psiFouriersol}), set
$$
  \widehat{g} = (\mathbf{L}_q-\mathbf{L}_0)\widehat{\psi|_{\Omega}}
$$
and define a function $g:\partial\Omega\rightarrow \mathbb{C}$ using the truncated Fourier series inversion:
$$
  g(\theta) = \sum_{n=-N}^{N} \widehat{g}(n)\varphi_n(\theta).
$$
Then the scattering transform can be computed approximately using the formula
\begin{equation}\label{approxT}
  \mathbf{t}(k) \approx  \int_0^{2\pi} e^{i\overline{k}\exp(-i\theta)}g(\theta) \,d\theta.
\end{equation}
The approximation in (\ref{approxT}) is most accurate for $k$ near zero.

Of course, the accuracy in (\ref{approxT}) can be improved also by increasing $N$, but there is a limit to that. With reasonable computational resources it is possible to compute $\mu_n(q)$ accurately enough up to $|n|\leq N=16$, but computation for $|n|>16$ quickly becomes an extremely hard problem. The reason is that the difference $|\mu_n(q)-|n||$ becomes exponentially small as $|n|$ grows. We remark that one can achieve higher accuracy by computing the difference $|\mu_n(q)-|n||$ directly analogously to the approach in \cite{HH}, but in this work there was no need for that.

\section{Computational results}\label{sec:compres}

\noindent
We will study two examples numerically: the simple case $q_\lambda=\lambda w$ in Section \ref{sec:firstexample} and a more complicated case $\widetilde{q}_\lambda=\widetilde{q}_0+\lambda w$ in Section \ref{sec:secondexample}. Here $\widetilde{q}_0$ is a nontrivial conductivity-type potential.

\subsection{Definition of a test function $w$}

We define a radial $C^2_0$ function $w(z)=w(|z|)$ as follows. Take two radii $0<R_1<R_2<1$ and define $w(|z|)$ in three pieces:
\begin{equation}\label{testfunw}
w(|z|)=
\left\{
\begin{array}{ccl}
1&\mbox{for}& 0\leq |z|\leq R_1,\\
p(|z|)&\mbox{for}& R_1<|z|<R_2,\\
0&\mbox{for}& R_2\leq |z|\leq 1.
\end{array}
\right.
\end{equation}
The polynomial $p$ in (\ref{testfunw}) is constructed as follows. Note that the polynomial $\widetilde{p}(t)=1-10t^3+15t^4-6t^5$ is smooth in the interval $0\leq t\leq 1$ and satisfies $p^{\prime\prime}(0)=p^{\prime}(0)=0=p^{\prime}(1)=p^{\prime\prime}(1)$. Set for $R_1\leq t\leq R_2$
$$
  p(t) = \widetilde{p}(\frac{t-R_1}{R_2-R_1}).
$$
The test function defined above is twice continuously differentiable instead of infinitely smooth as in the theoretical part above. However, the discrepancy is not essential in this numerical work. See Figure \ref{fig:testfunw} for a plot of the test function $w$.
\begin{figure}
\begin{picture}(300,180)
\put(-70,-55){\includegraphics[width=15cm]{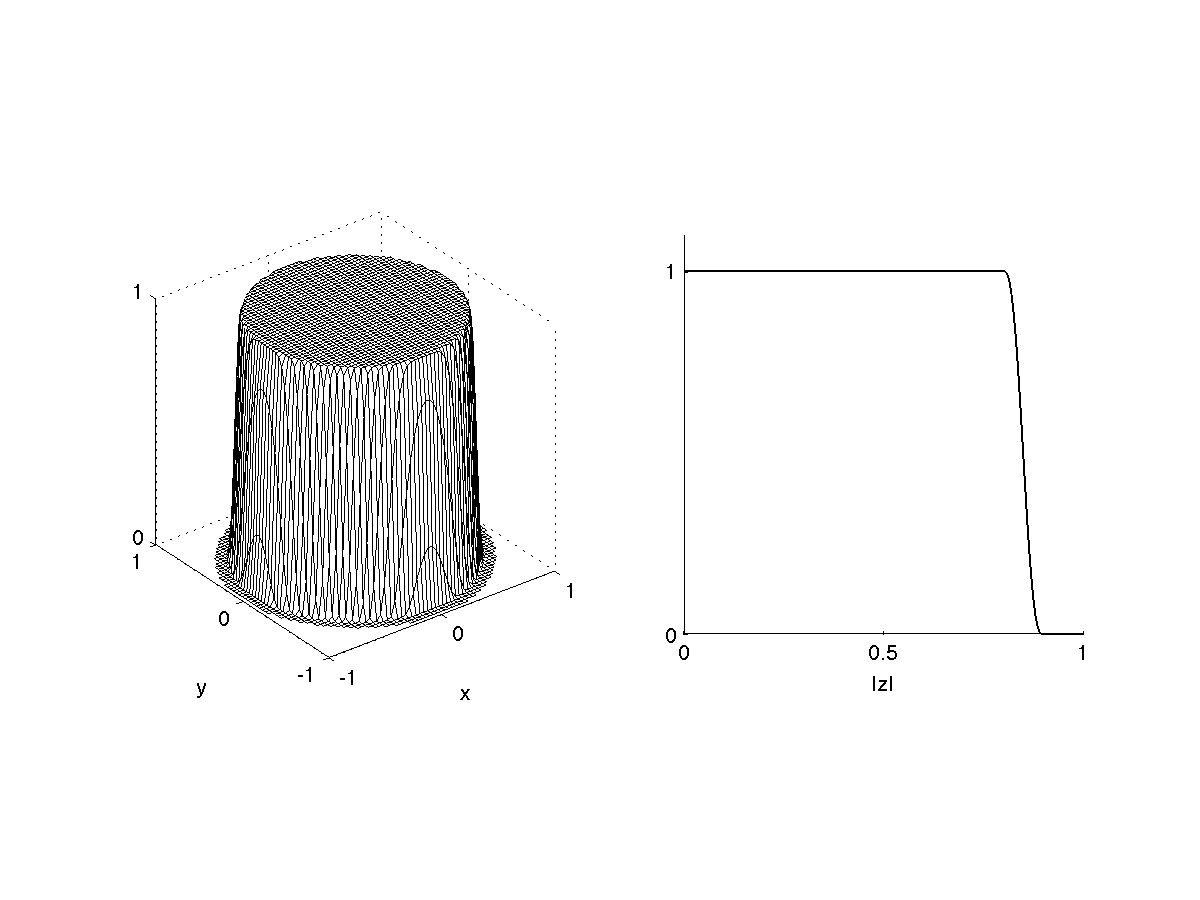}}
\put(200,120){\small Profile of $w(|z|)$}
\end{picture}
\caption{\label{fig:testfunw}Test function $w(z)=w(|z|)$ defined by formula (\ref{testfunw}) with $R_1=0.8$ and $R_1=0.9$.}
\end{figure}

\subsection{First example: zero potential at $\lambda=0$}\label{sec:firstexample}

We set $q_\lambda=\lambda w$ with the radial test function defined by (\ref{testfunw}) with $R_1=0.8$ and $R_1=0.9$.

Our aim is to compute the radial scattering transform $\mathbf{t}_\lambda(|k|)$ for $0<|k|\leq 3.5$ and for the parameter $\lambda$ ranging in a suitable interval. In practice we choose the following finite set of $k$-values:
\begin{equation}\label{kgrid}
  k=0.01,0.02,0.03,\dots,3.49,3.50.
\end{equation}
Note that the $k$-grid is bounded away from zero by a significant gap of size $10^{-2}$. Furthermore, we consider the following choices of parameter $\lambda$:
\begin{equation}\label{lambdagrid}
  \lambda=-35.00,-34.95,-34.90,\dots,34.90,34.95,35.00.
\end{equation}

We start the numerical computations by constructing the matrices (\ref{matrixeq}) for each $k$-value listed in (\ref{kgrid}). We take $N=12$, so each $\mathbf{H}_k$ has size $25{\times}25$. These matrices need to be computed only once for a given $k$, so we can reuse the matrices in our second example below.

Next we use the methods described in Section \ref{sec:numericalevals} to compute the eigenvalues of the {\sc dn} map corresponding to each potential $q_\lambda$ with $\lambda$ ranging as in (\ref{lambdagrid}). We construct a finite element mesh for the unit disc with 131585 nodes and 262144 triangles. We use Matlab's PDE toolbox to solve the Neumann problem (\ref{NeumProb}) with $g=\varphi_n$ for $1,\dots,N$, and get accurate approximations to the eigenvalues $\nu_1,\dots,\nu_N$ of the {\sc nd} map. By (\ref{eigiden2}) and (\ref{Nevals}) we see that we know all eigenvalues $\mu_n(q_\lambda)$ of the {\sc dn} maps except for $\mu_0(q_\lambda)$. We use FEM to solve Dirichlet problems of the form (\ref{eq:DP}) where $q=q_\lambda$ and $f=1$; this way we get good approximations to $\mu_0(q_\lambda)$. See Figure \ref{fig:qzeroevals} for plots of the eigenvalue $\mu_0(q_\lambda)$ as function of the parameter $\lambda$.

\begin{figure}
\begin{picture}(300,180)
 \put(-60,15){\includegraphics[width=7cm]{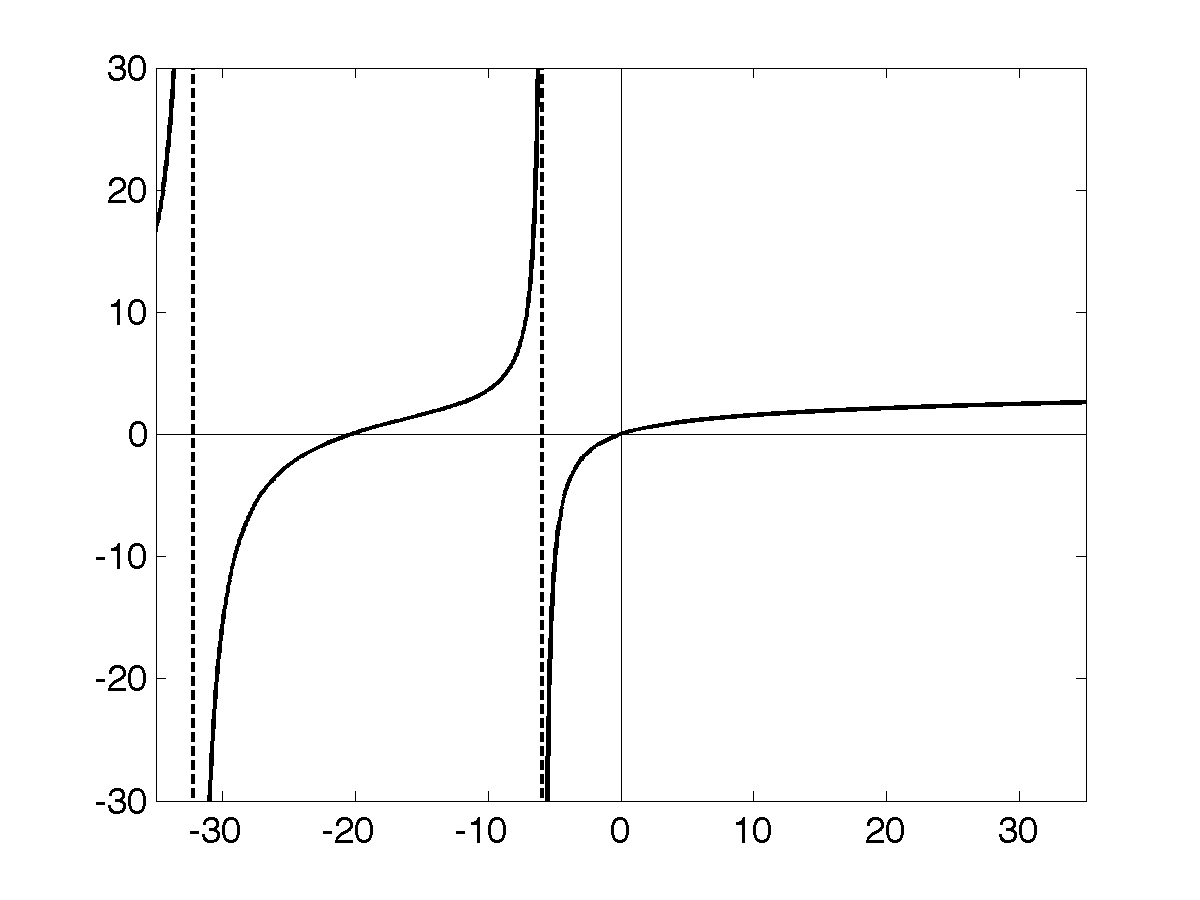}}
 \put(160,15){\includegraphics[width=7cm]{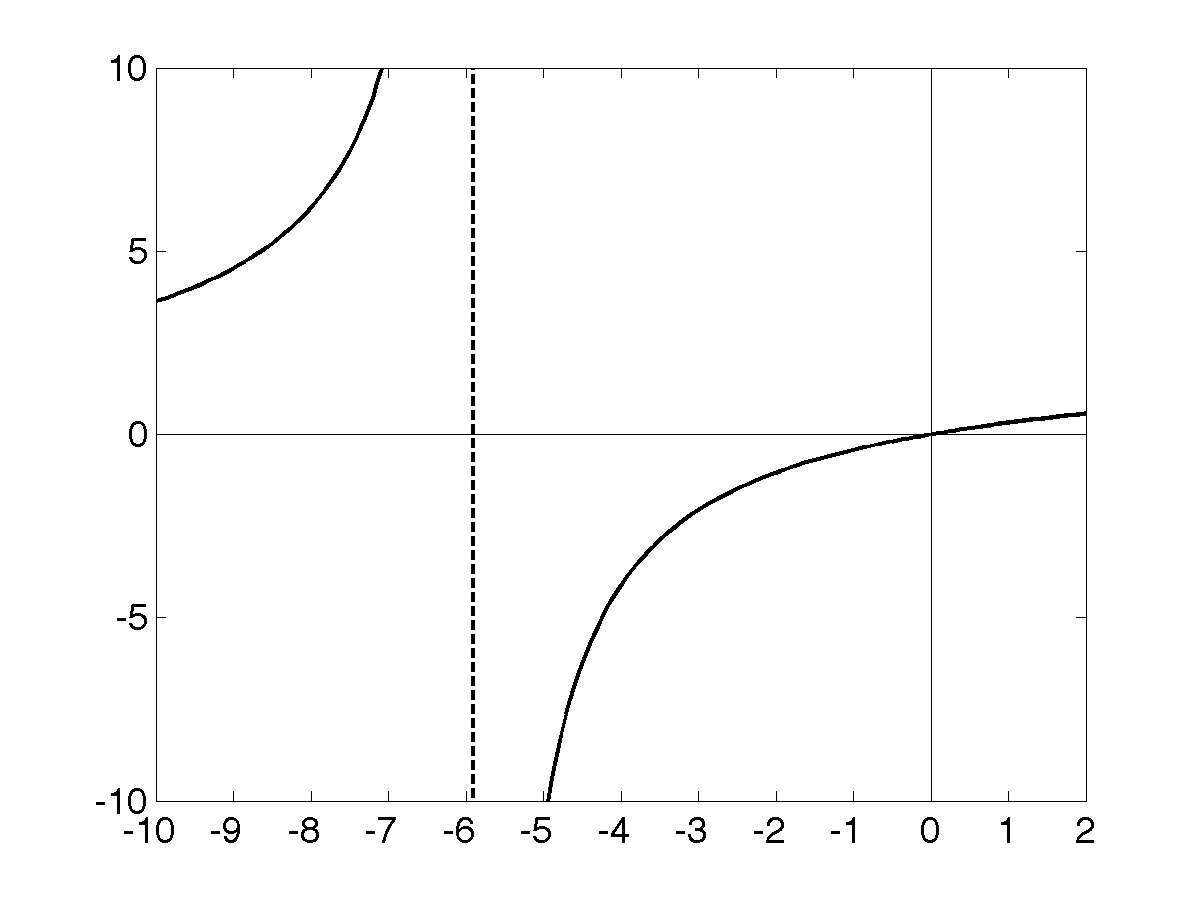}}
\put(42,0){$\lambda$}
\put(262,0){$\lambda$}
\end{picture}
\caption{\label{fig:qzeroevals}Eigenvalues $\mu(\lambda)=\mu_0(q_\lambda)$ corresponding to the first example potential. Left: plot with full parameter range $-35\leq \lambda\leq 35$. Right: detail of the left plot. Note that $\mu(0)=0$ and $\mu^\prime(0)>0$ as predicted by Lemma \ref{lemma:mu}. Also, note that Dirichlet eigenvalues of the potential $q_\lambda$ in the unit disc cause singularities in $\mu(\lambda)$.}
\end{figure}

Now we have all the ingredients for solving the matrix equation (\ref{def:psiBIE}) and evaluating the scattering transform by (\ref{approxT}). We used Matlab, a parallelization middleware solution provided by Techila Ltd, and two hardware solutions: the ``Ukko'' cluster computer of the Computer Science Department of University of Helsinki, and Microsoft Azure cloud computing services.

See Figure \ref{fig:scat1A} for a plot of the scattering transform profiles as a grayscale plot, and Figure \ref{fig:scatmaps1} for some selected profiles as conventional plots. Further, see Figure \ref{fig:asympfit} for a comparison of numerical results and the asymptotic formula (\ref{eq:r.lambda}) for the radius of the exceptional circle.

\begin{figure}
\begin{picture}(300,325)
 \put(-60,20){\includegraphics[width=14cm]{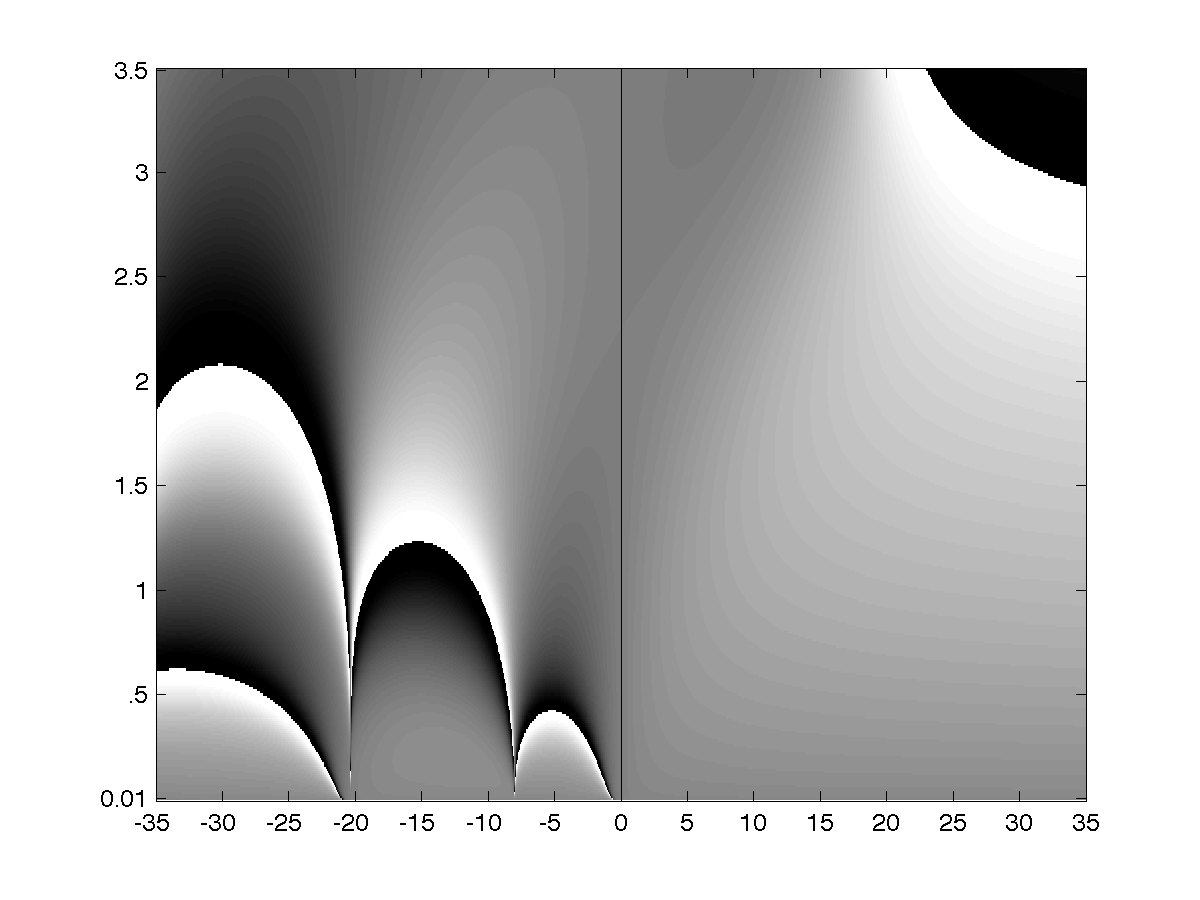}}
\put(140,0){\Large $\lambda$}
\put(-75,200){\Large $|k|$}
\end{picture}
\caption{\label{fig:scat1A}Scattering transform corresponding to the first example. The horizontal axis is the parameter $\lambda$ in the definition $q_\lambda(z) = \lambda w(z)$ of the potential. The vertical axis is $|k|$. There are curves along which a singular jump ``from $-\infty$ to $+\infty$'' appears.
The $k$ values at those curves are exceptional points. See Figure \ref{fig:scatmaps1} for further illustration of the singularities.}
\end{figure}

\begin{figure}
\begin{picture}(300,480)
 \put(190,475){\small Profile of scattering transform}
 \put(190,370){\includegraphics[height=3.5cm]{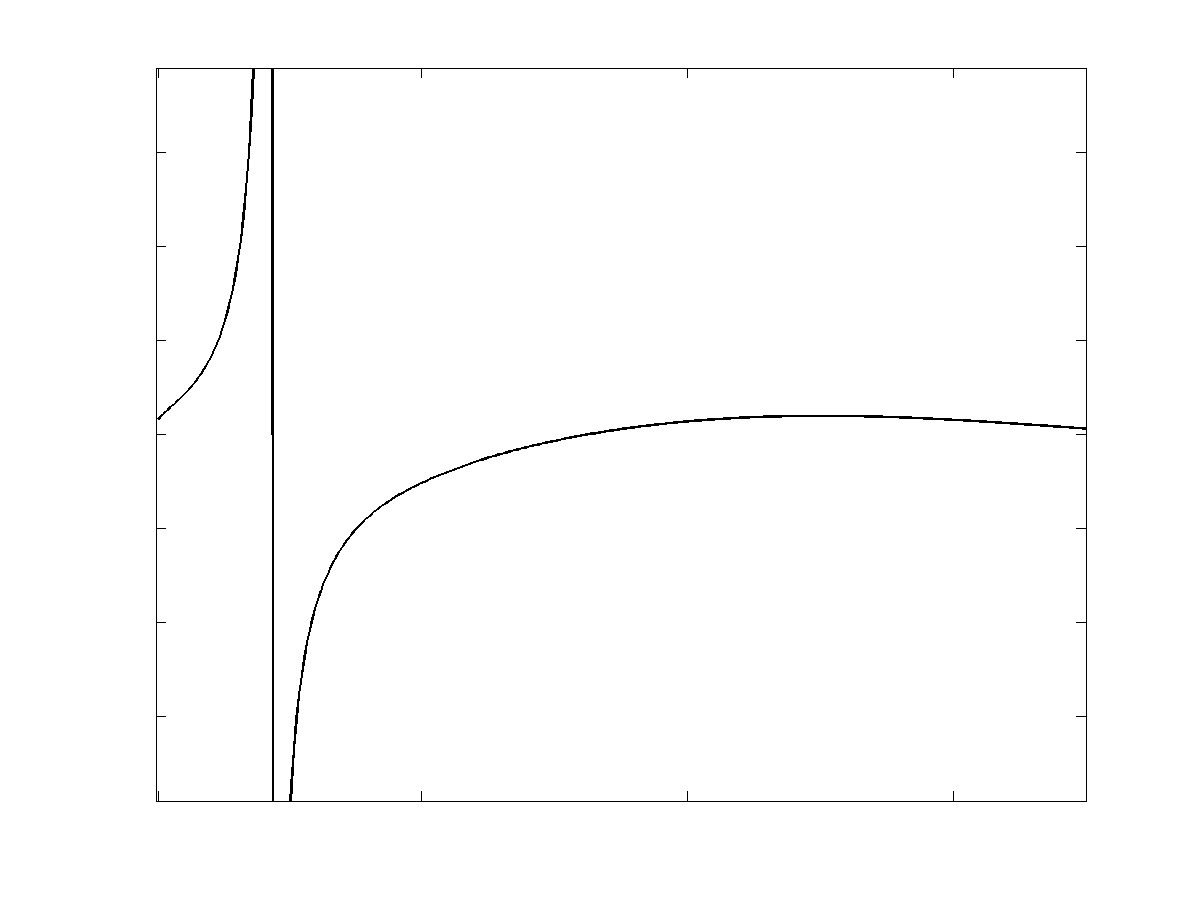}}
 \put(190,250){\includegraphics[height=3.5cm]{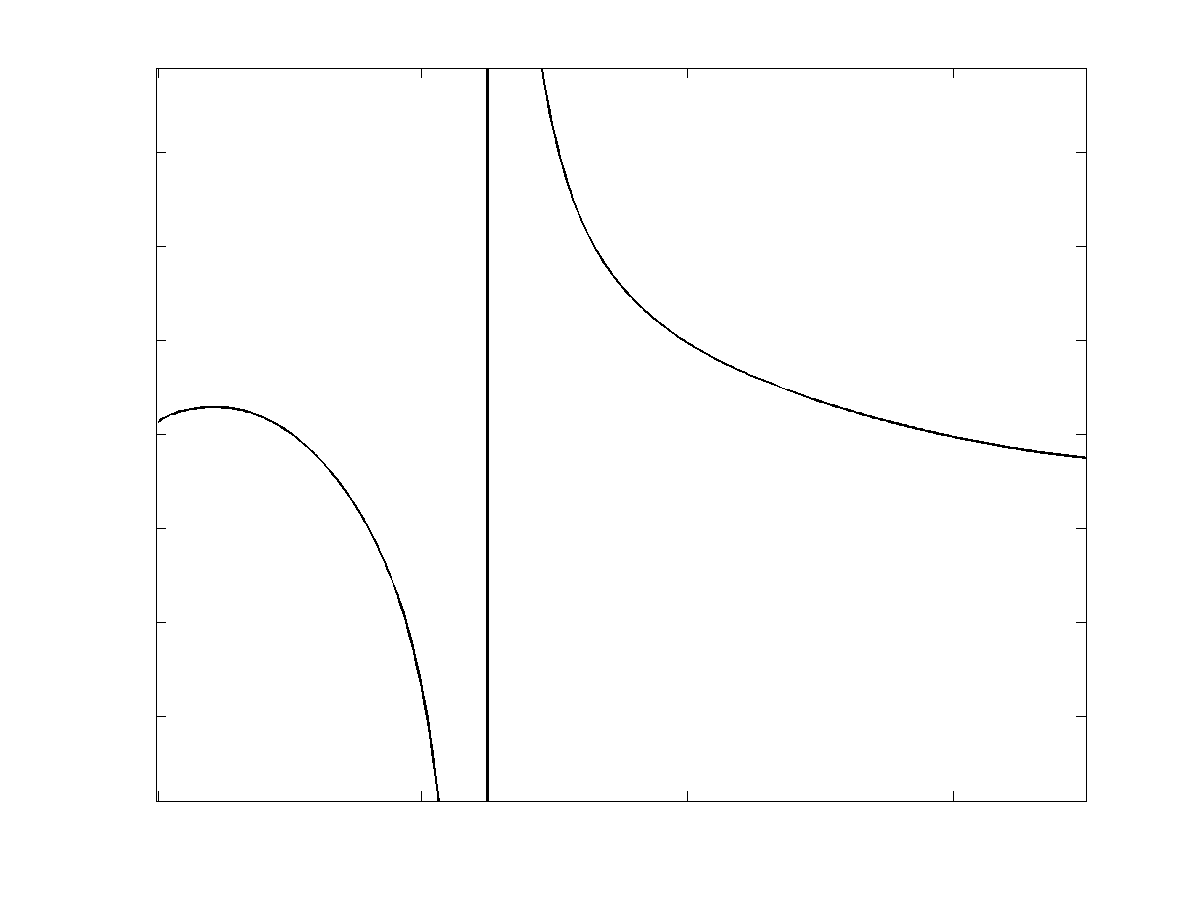}}
 \put(190,130){\includegraphics[height=3.5cm]{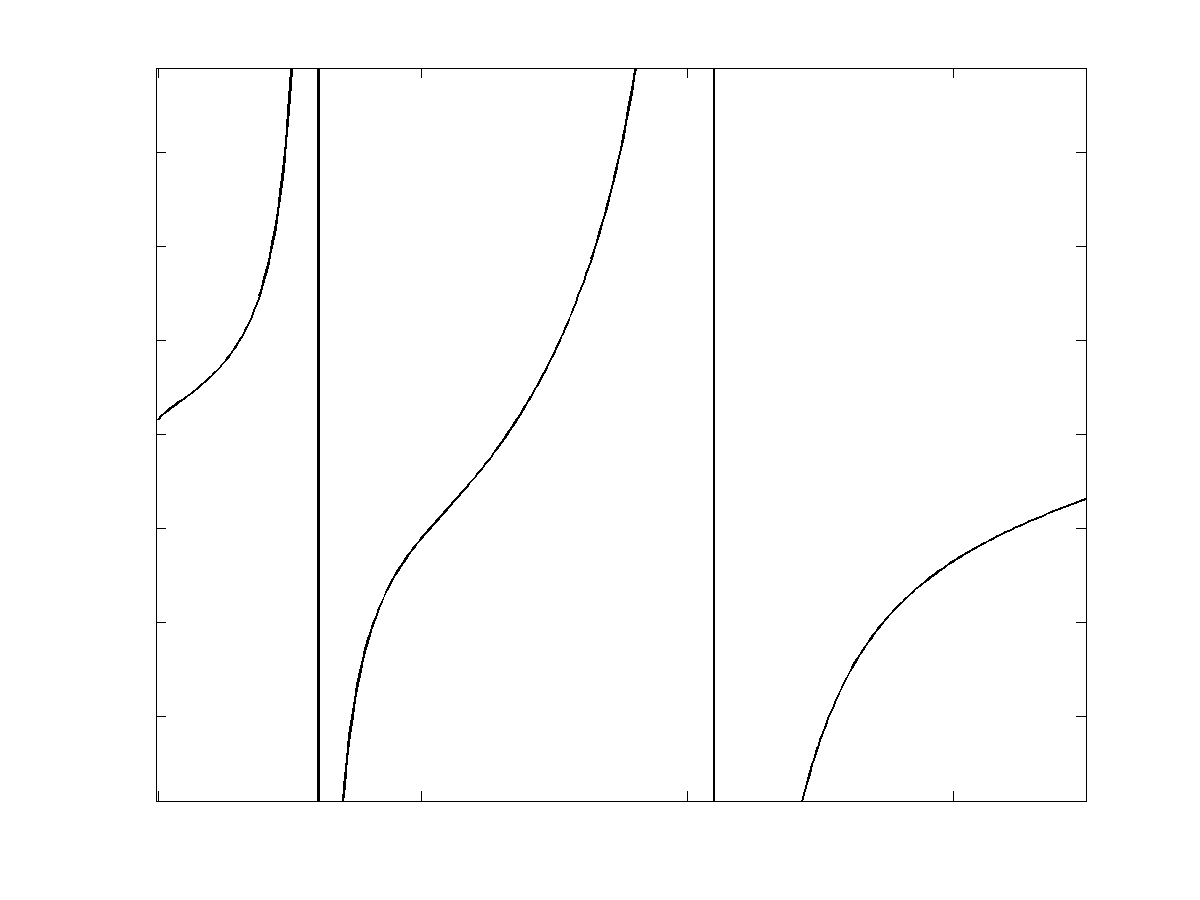}}
 \put(190,10){\includegraphics[height=3.5cm]{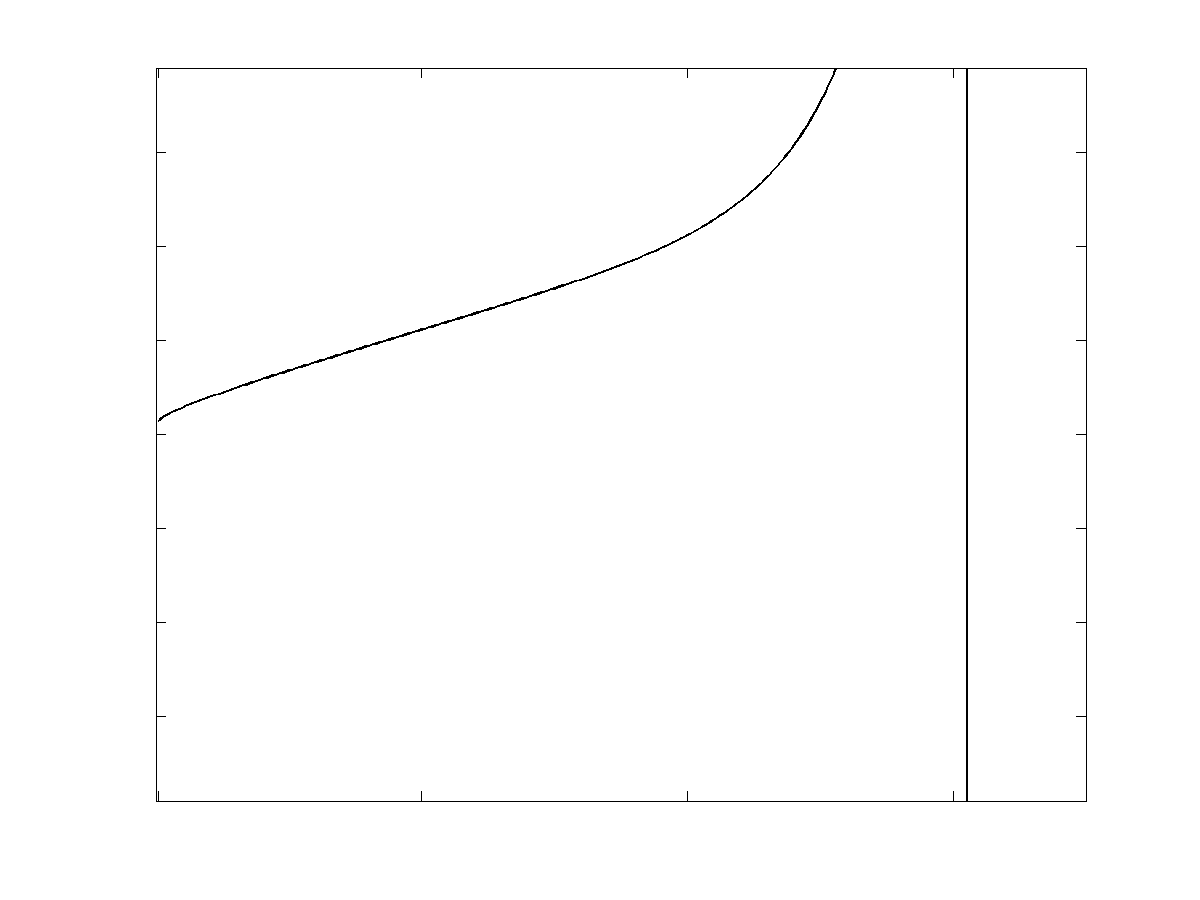}}
 \put(240,-8){\small $|k|$}
\put(-40,410){$\lambda=-5$}
\put(-40,290){$\lambda=-15$}
\put(-40,170){$\lambda=-30$}
\put(-40,50){$\lambda=30$}
 \put(10,370){\includegraphics[height=3.5cm]{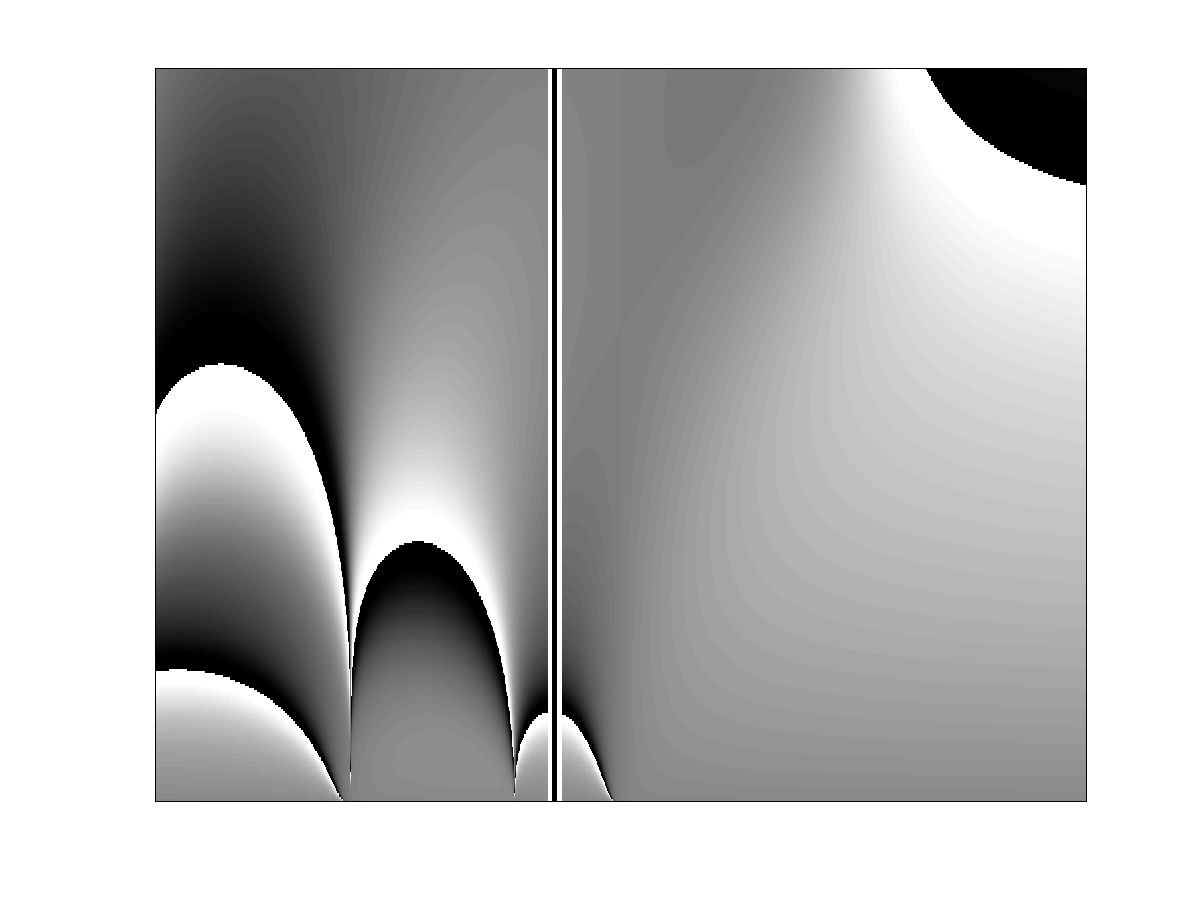}}
 \put(10,250){\includegraphics[height=3.5cm]{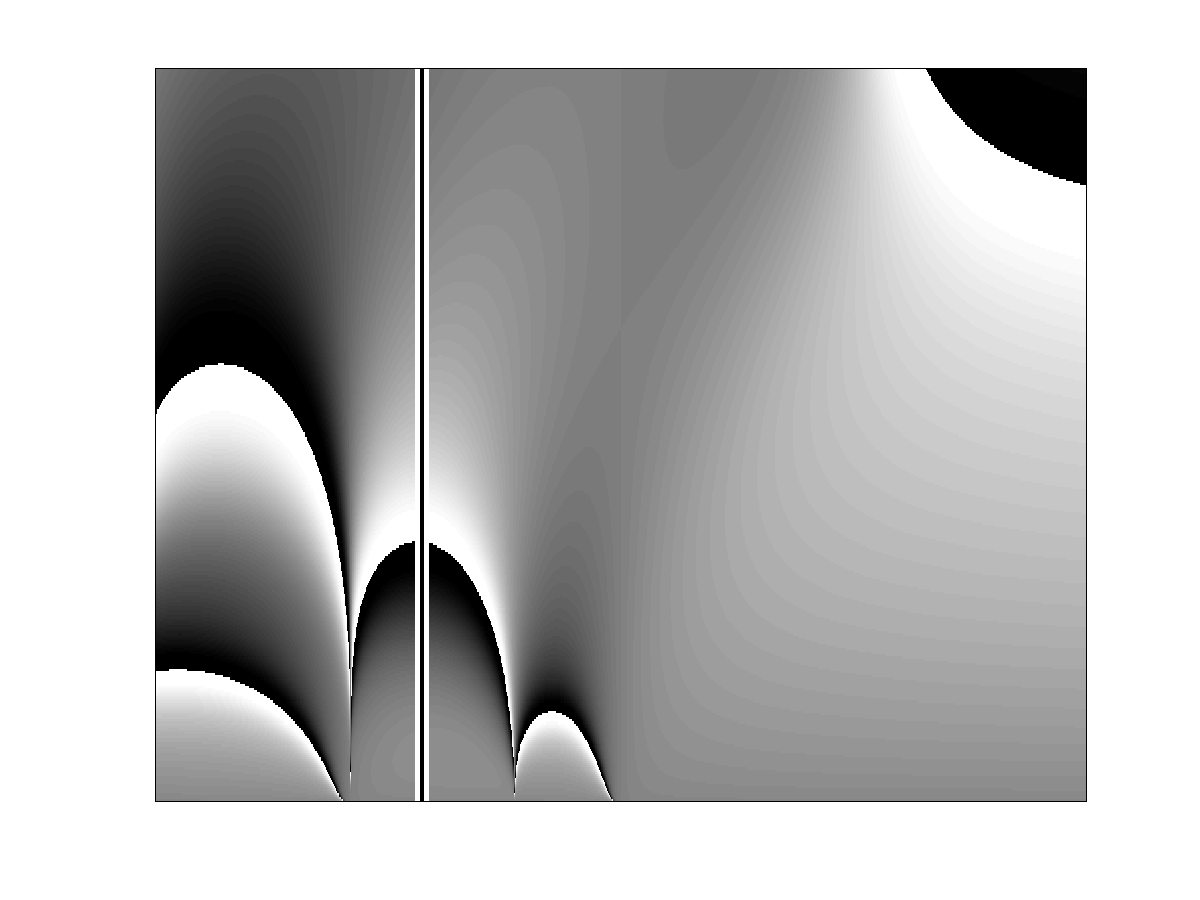}}
 \put(10,130){\includegraphics[height=3.5cm]{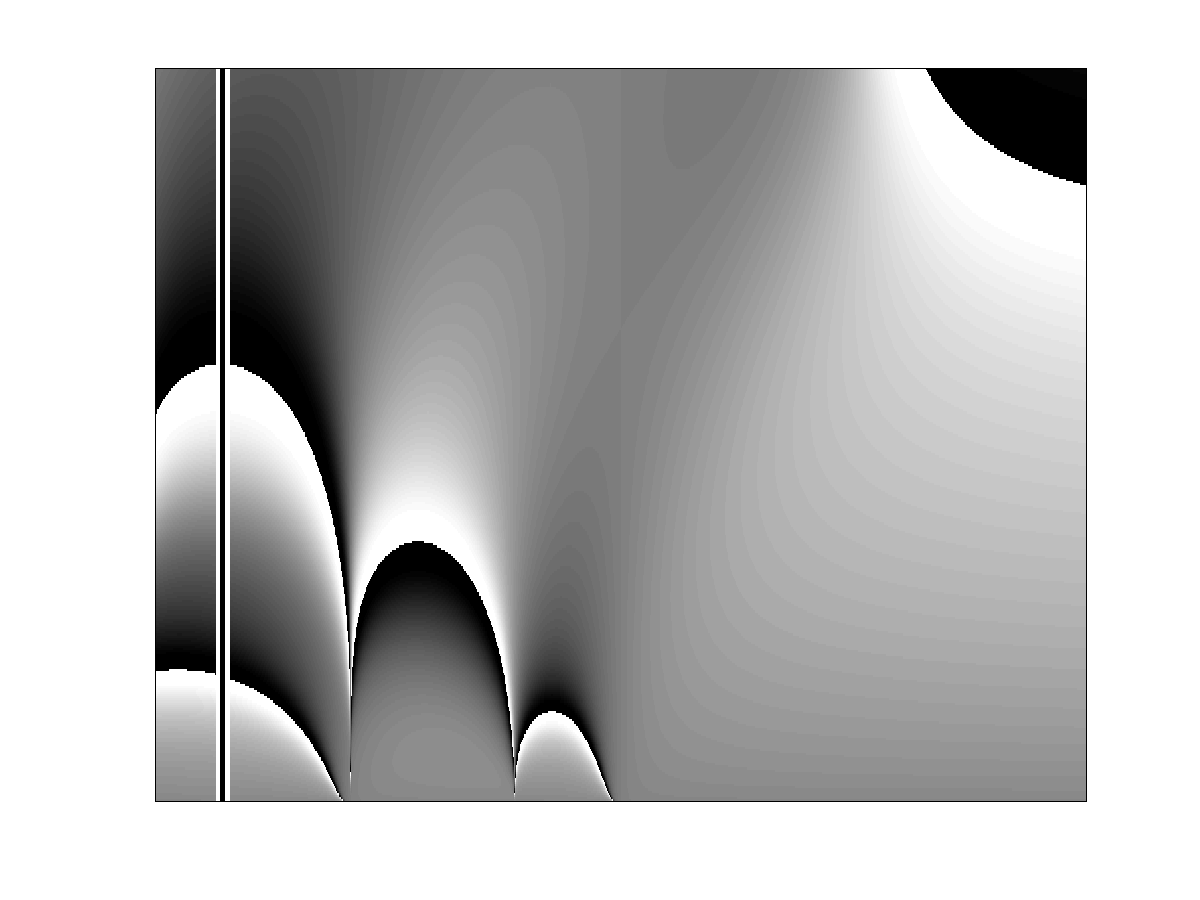}}
 \put(10,10){\includegraphics[height=3.5cm]{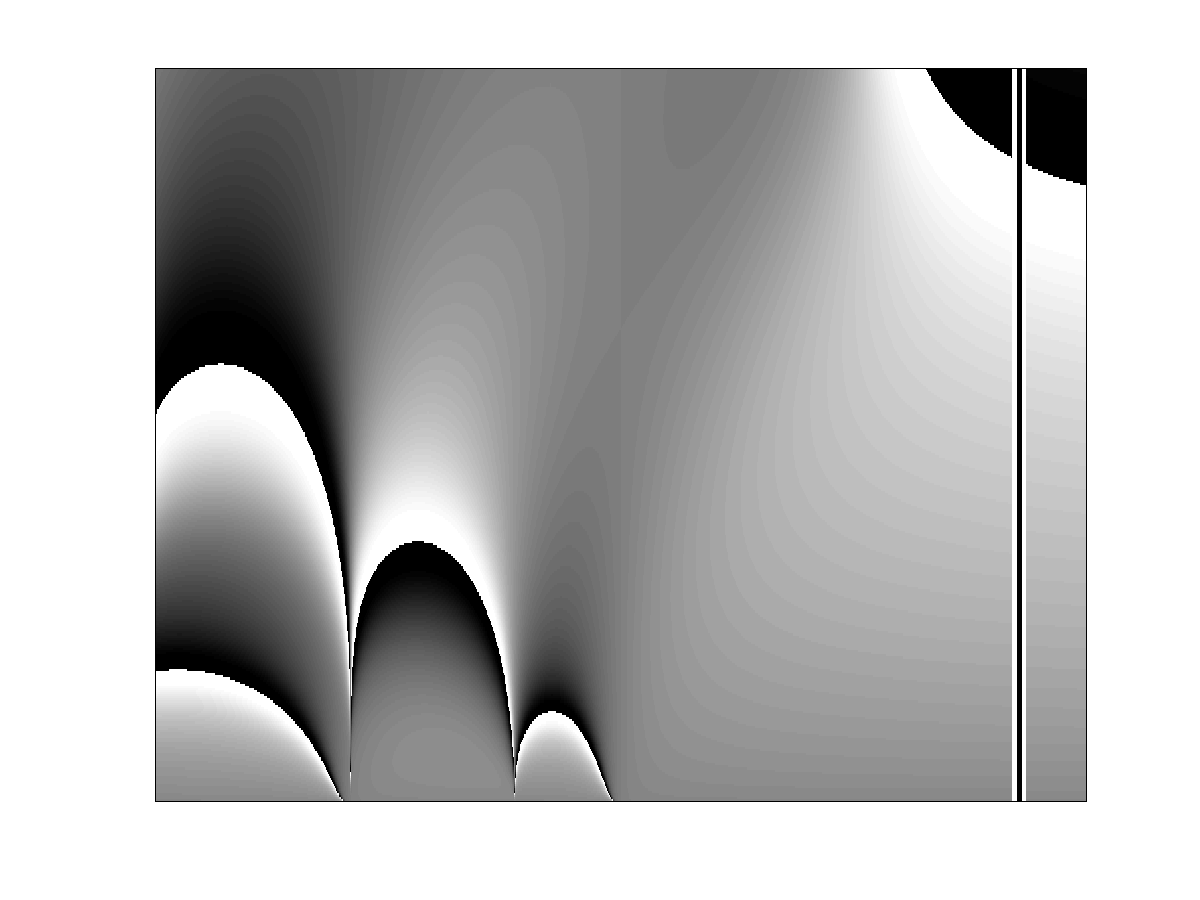}}
\end{picture}
\caption{\label{fig:scatmaps1}Profiles of scattering transforms for various $\lambda$. This picture corresponds to Example 1.}
\end{figure}
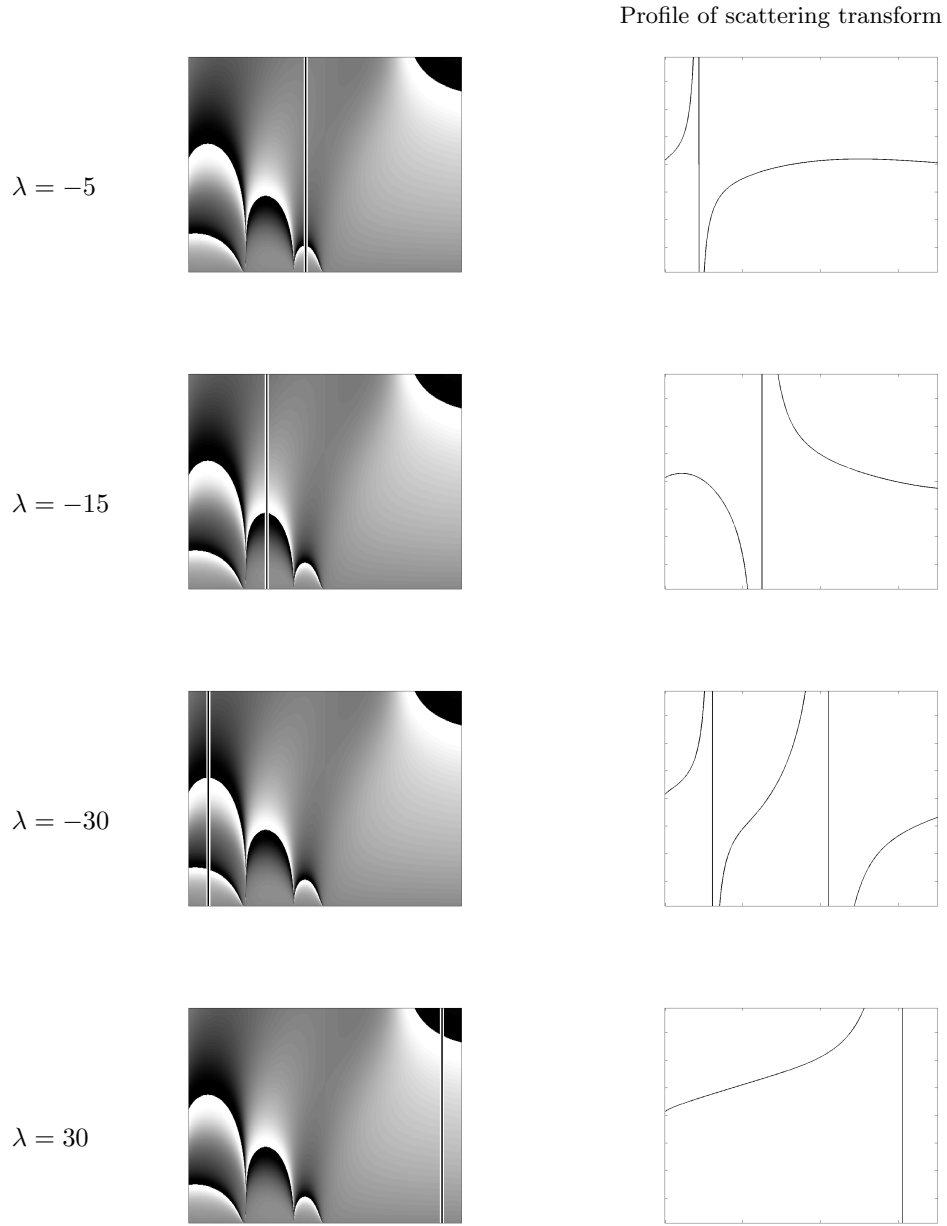

\begin{figure}
\begin{picture}(400,160)
 \put(-20,10){\includegraphics[height=4.6cm]{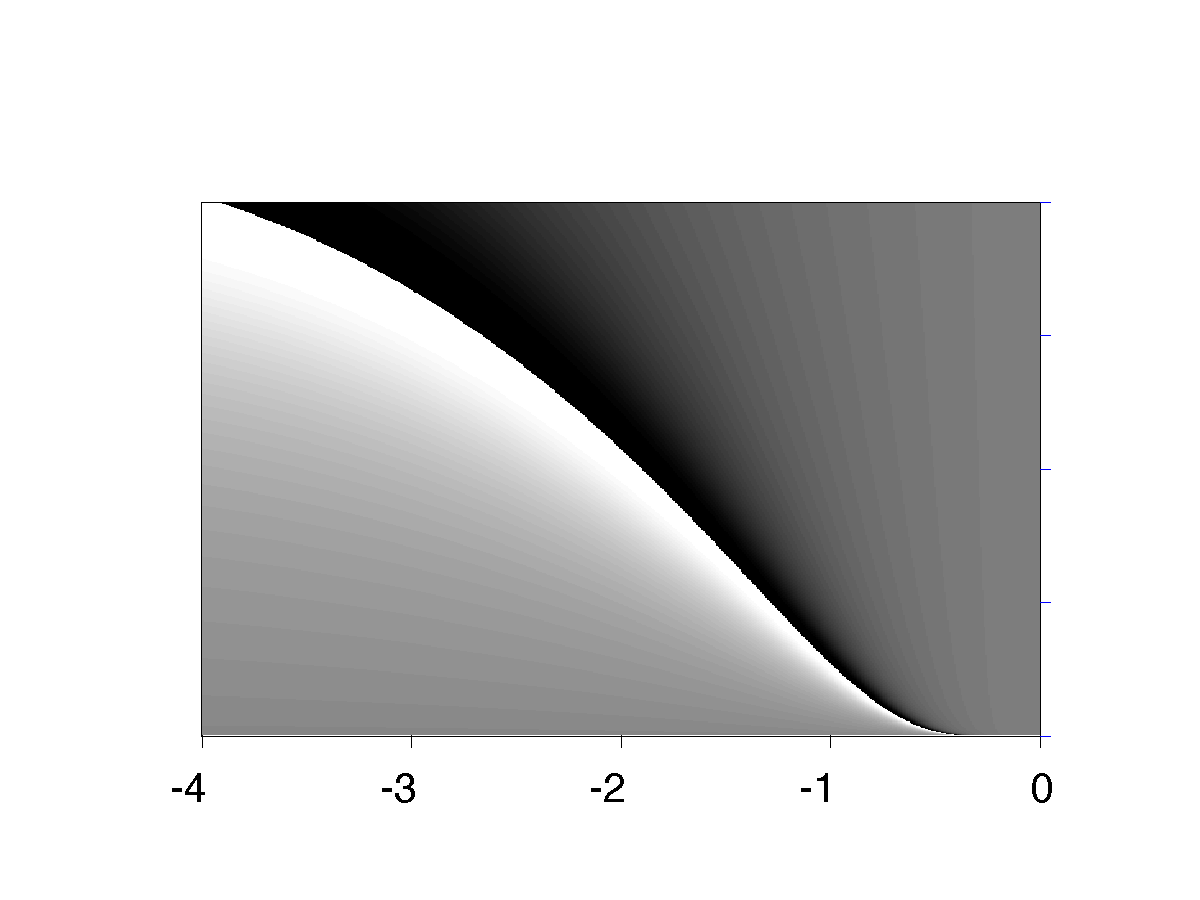}}
 \put(175,10){\includegraphics[height=4.8cm]{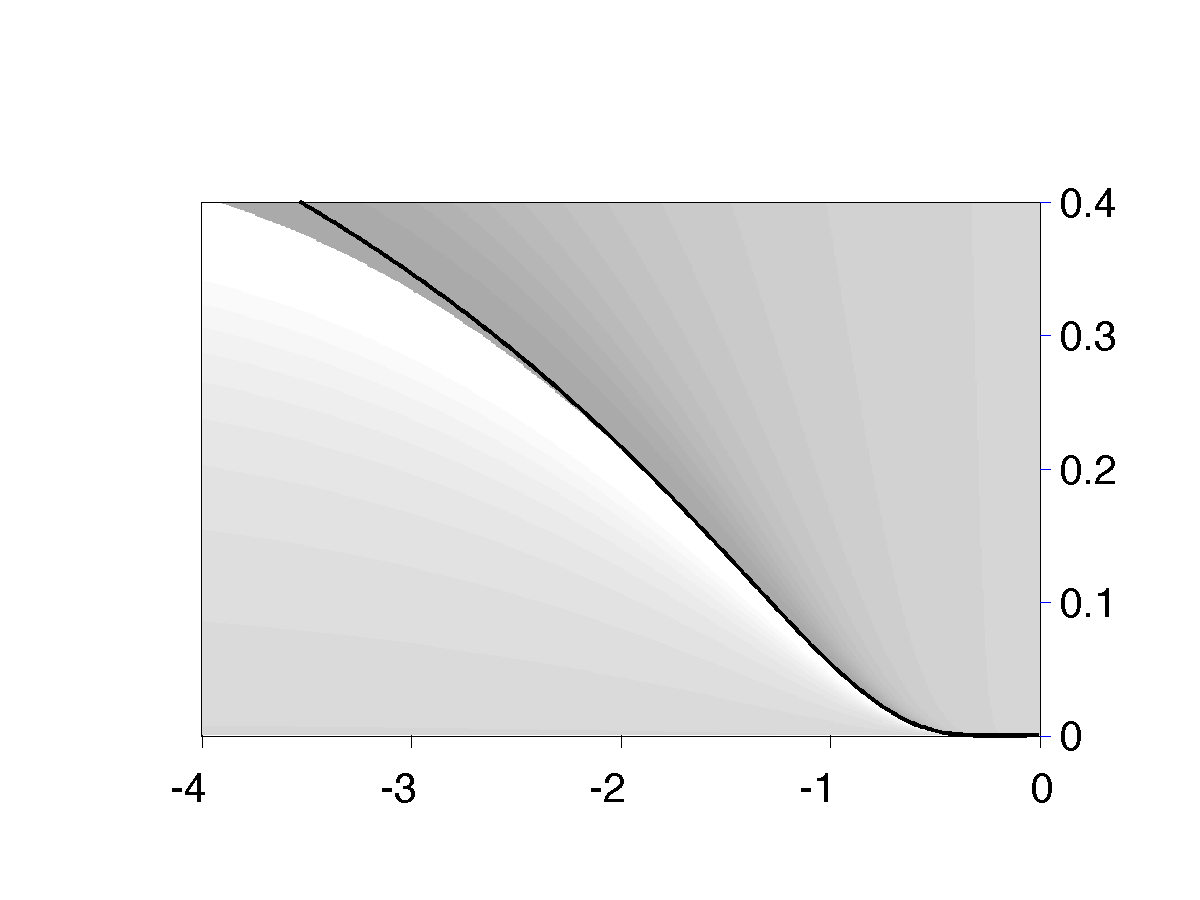}}
 \put(377,87){\large $|k|$}
 \put(105,0){$\lambda$}
 \put(300,0){$\lambda$}
 \put(220,130){$\displaystyle \exp\left[  2\pi\left(  h+\frac{1}{2\pi\mu(\lambda)}\right)  \right] $}
\end{picture}
\caption{\label{fig:asympfit}Comparison of numerical results and the asymptotic formula (\ref{eq:r.lambda}) for the radius of the exceptional circle. This plot is for Example 1. Left: detail from Figure \ref{fig:scat1A} with parameters ranging in the rectangle $-4\leq\lambda\leq 0$ and $0.001\leq |k|\leq 0.4$. Right: the asymptotic function $r(\lambda)$ given by Theorem \ref{thm:main}. For ease of comparison, we also show in the background the pixel image from the left but with a lighter colormap.  The asymptotic formula matches the computational result very closely in the interval $-2\leq\lambda\leq 0$.}
\end{figure}
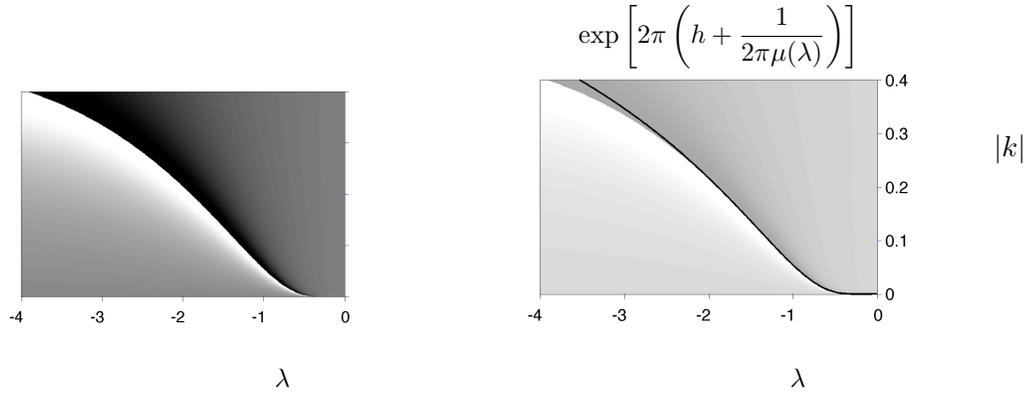

\clearpage

\subsection{Second example: nontrivial conductivity-type potential at $\lambda=0$}\label{sec:secondexample}

Here we define ${\widetilde{q}}_\lambda={\widetilde{q}}_0+\lambda w$, where the test function $w$ is defined by formula (\ref{testfunw}) and plotted in Figure \ref{fig:testfunw}. The potential ${\widetilde{q}}_0$ corresponding to $\lambda=0$ is defined by ${\widetilde{q}}_0=\sigma^{-1/2}\Delta\sigma^{1/2}$. See Figure \ref{fig:sigmapoten} for plots of the conductivity $\sigma$ and the potential ${\widetilde{q}}_0$. Figure \ref{fig:sigmascat} shows the profile of the non-singular scattering transform of ${\widetilde{q}}_0$.

We compute the eigenvalues of the {\sc dn} map corresponding to each potential $\widetilde{q}_\lambda$ similarly than in Example 1. See Figure \ref{fig:qzeroevals2} for plots of the eigenvalue $\mu(\lambda)=\mu_0(\widetilde{q}_\lambda)$ as function of the parameter $\lambda$.

We compute the scattering transform similarly to Example 1 above. See Figure \ref{fig:scat2A} for a plot of the scattering transform profiles as a grayscale plot. Further, see Figure \ref{fig:asympfit2} for a comparison of numerical results and the asymptotic formula (\ref{eq:r.lambda}) for the radius of the exceptional circle.

\begin{figure}
\begin{picture}(390,400)
 \put(-30,190){\includegraphics[width=8cm]{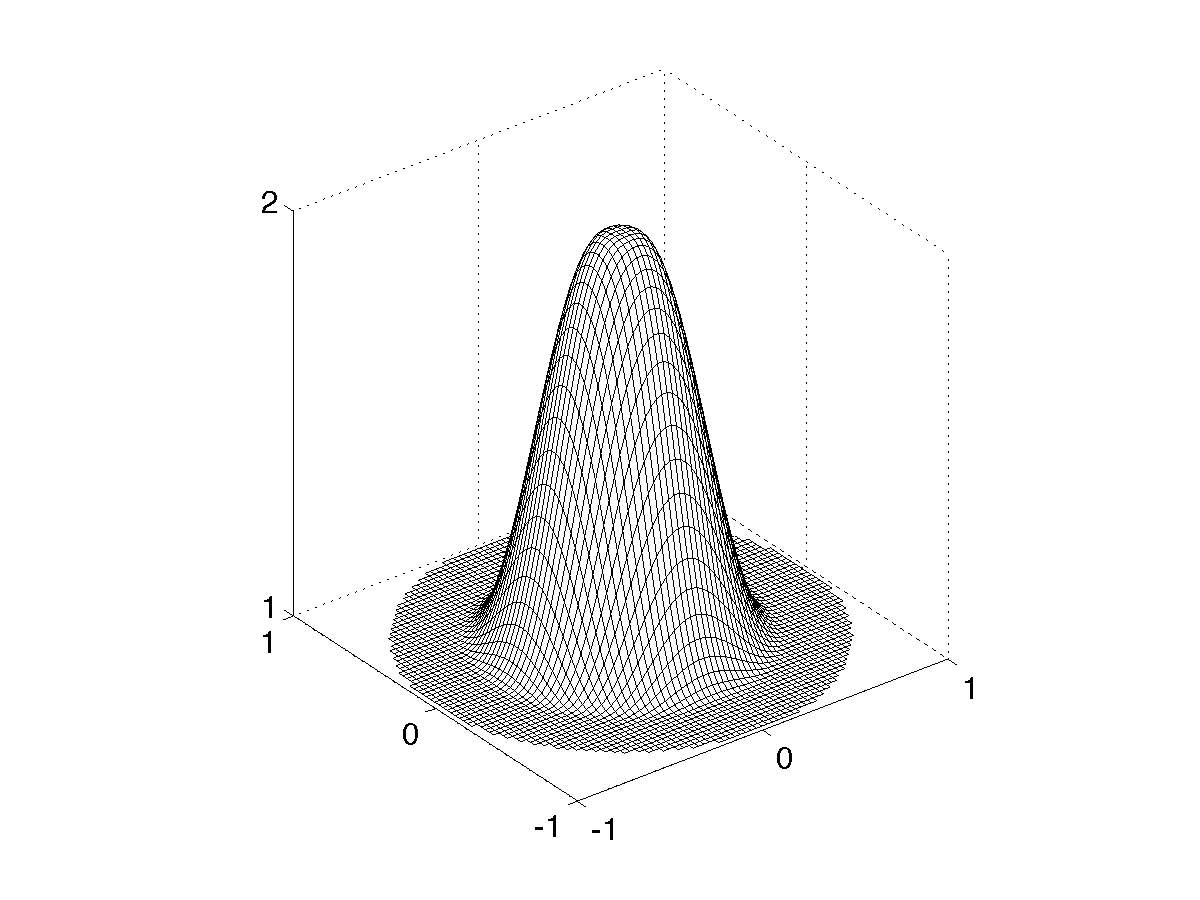}}
\put(-10,386){\small Conductivity $\sigma(z)$}
 \put(-30,0){\includegraphics[width=8cm]{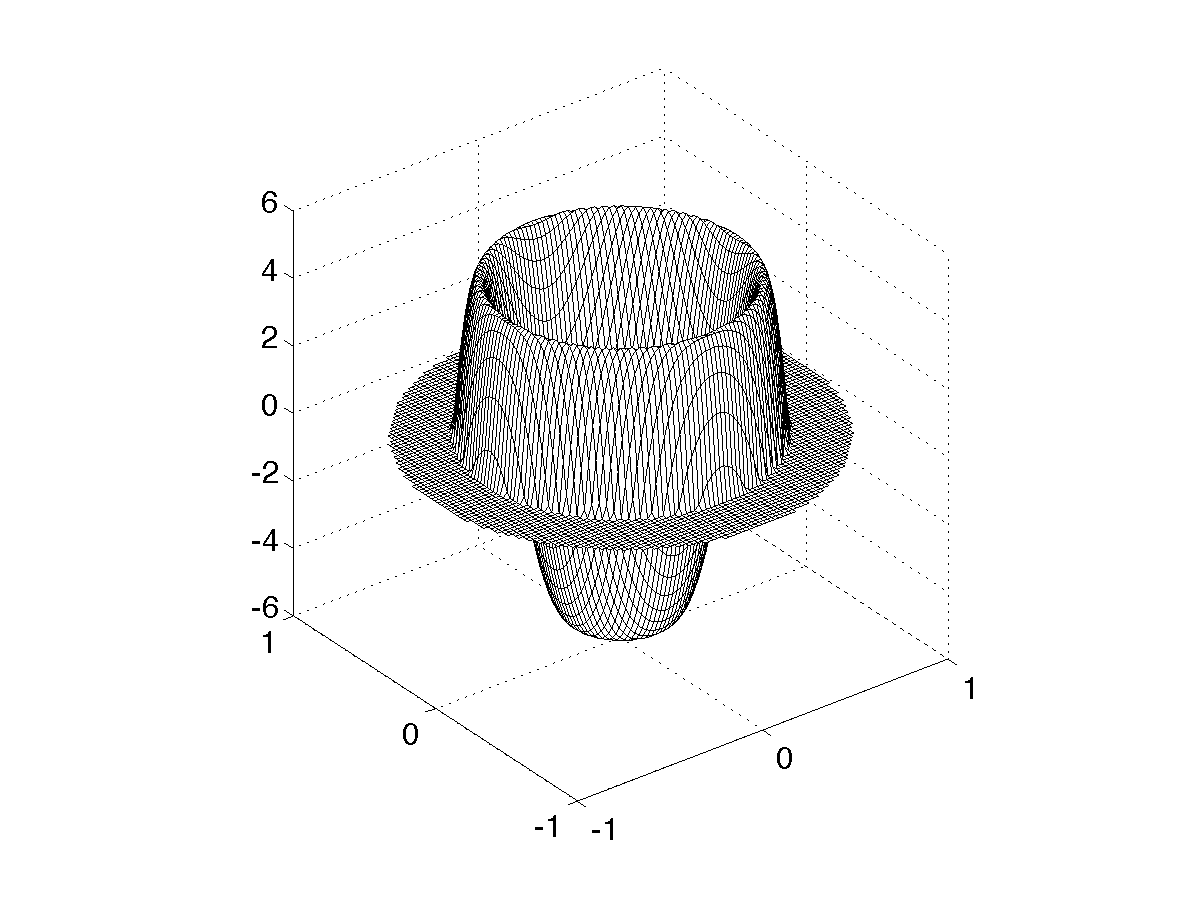}}
\put(-10,176){\small Conductivity-type potential ${\widetilde{q}}_0(z)$ }
 \put(203,220){\includegraphics[width=5.5cm]{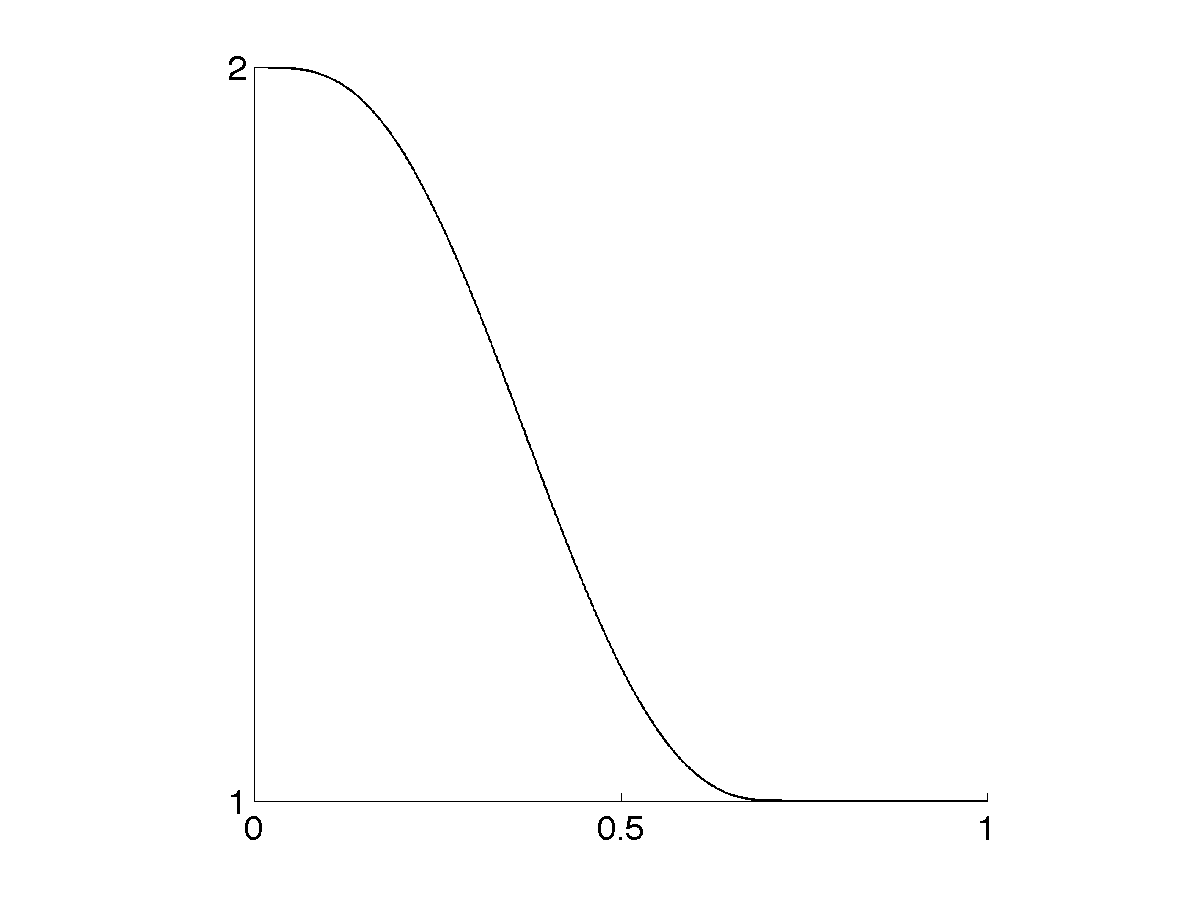}}
\put(200,386){\small Profile of conductivity }
\put(260,326){\small $\sigma(|z|)$}
 \put(200,10){\includegraphics[width=5.5cm]{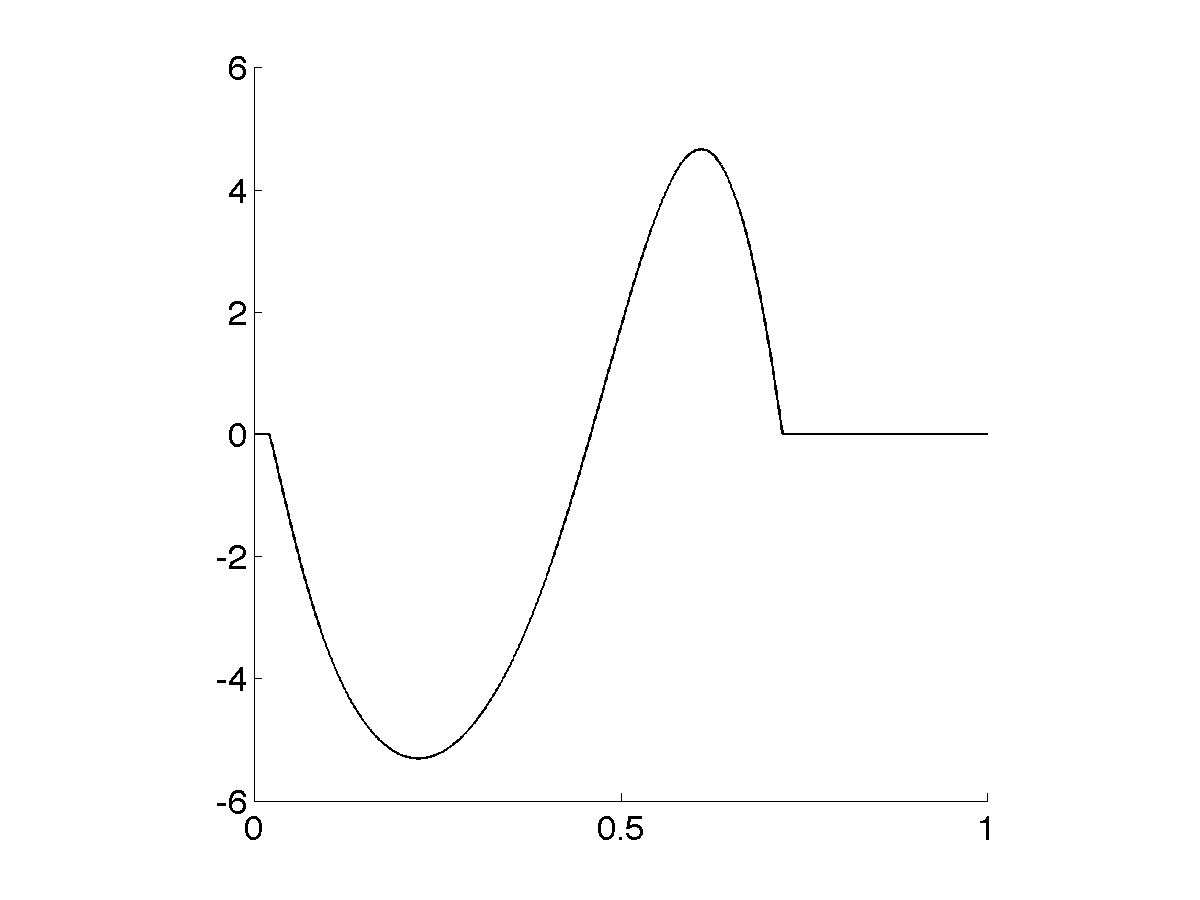}}
\put(200,176){\small Profile of potential }
\put(315,110){\small ${\widetilde{q}}_0(|z|)$}
\put(277,-2){\small $|z|$}
\put(277,208){\small $|z|$}
\end{picture}
\caption{\label{fig:sigmapoten}Top row: mesh plot and profile plot of the rotationally symmetric conductivity $\sigma(z)=\sigma(|z|)$. Bottom row: mesh plot and profile plot of the resulting conductivity-type potential ${\widetilde{q}}_0(z)={\widetilde{q}}_0(|z|)$.}
\end{figure}
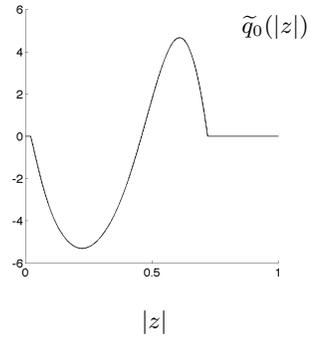

\begin{figure}
\begin{picture}(390,200)
 \put(-30,-10){\includegraphics[width=9cm]{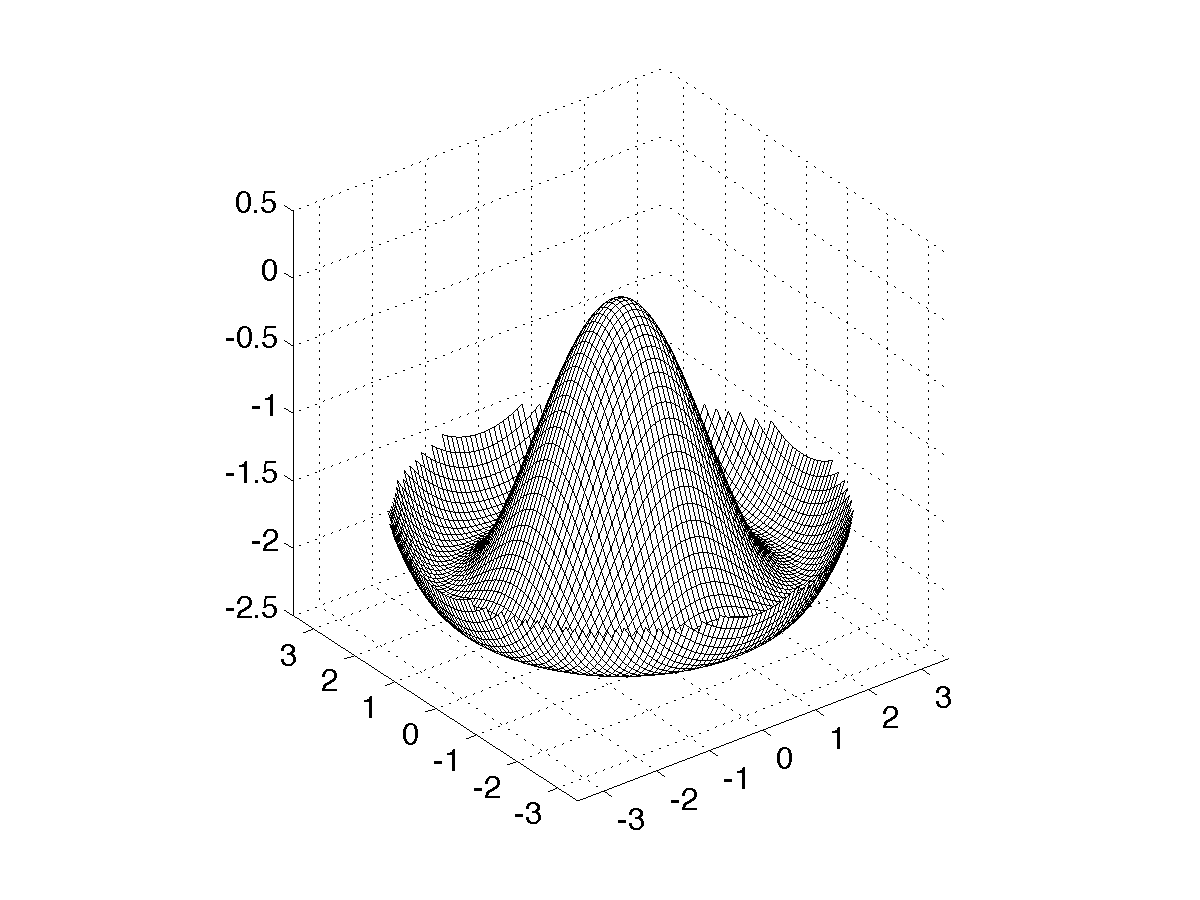}}
\put(-10,170){\small Scattering transform $\mathbf{t}_0(k)$ }
 \put(200,0){\includegraphics[width=6.5cm]{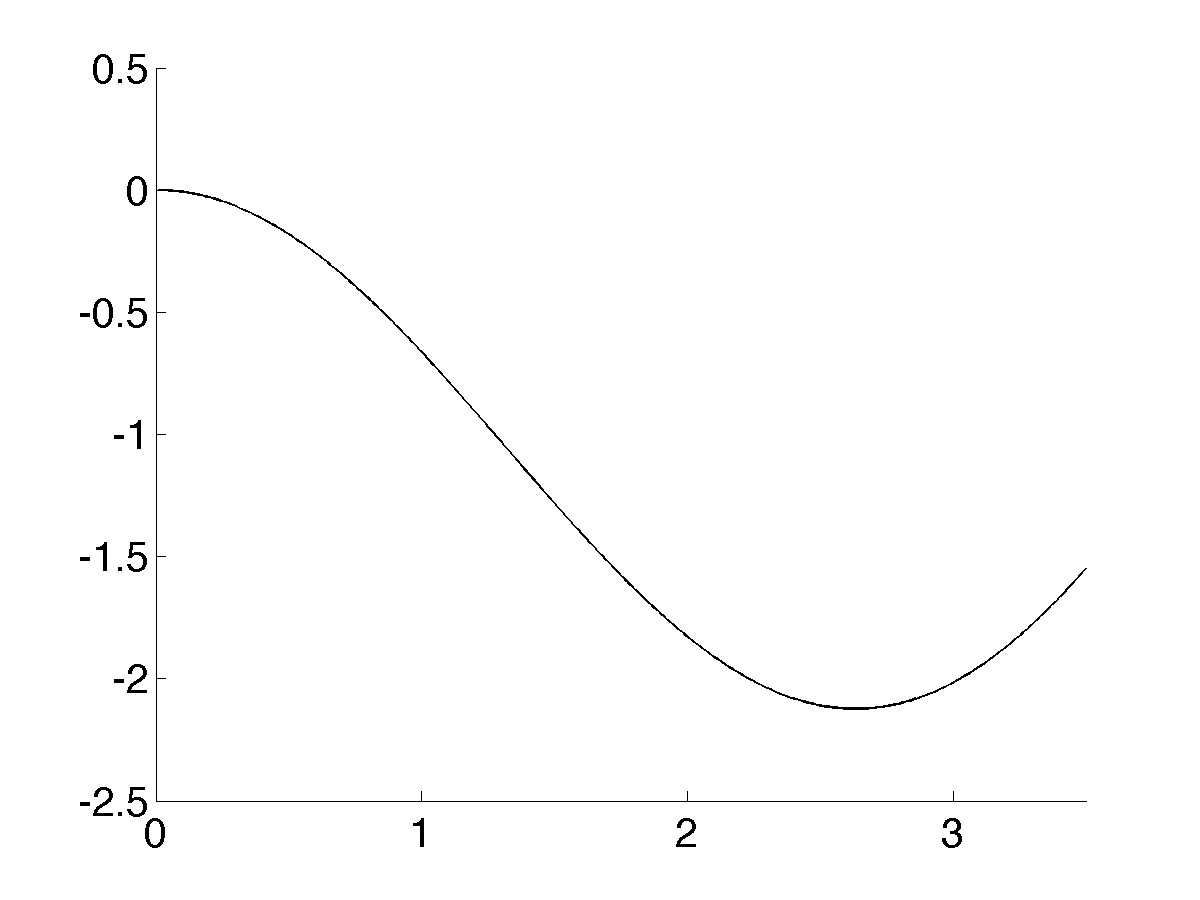}}
\put(200,170){\small Profile of scattering transform }
\put(240,120){ $\mathbf{t}_0(|k|)$}
\end{picture}
\caption{\label{fig:sigmascat}Left: mesh plot of the rotationally symmetric scattering transform $\mathbf{t}_0(k)=\mathbf{t}_0(|k|)$ corresponding to the initial potential ${\widetilde{q}}_0$ shown in Figure \ref{fig:sigmapoten}. Right: profile plot of $\mathbf{t}_0(|k|)$. }
\end{figure}
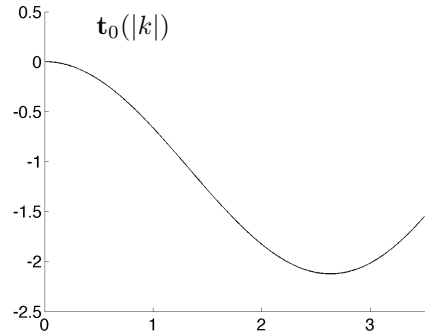

\begin{figure}
\begin{picture}(300,180)
 \put(-60,15){\includegraphics[width=7cm]{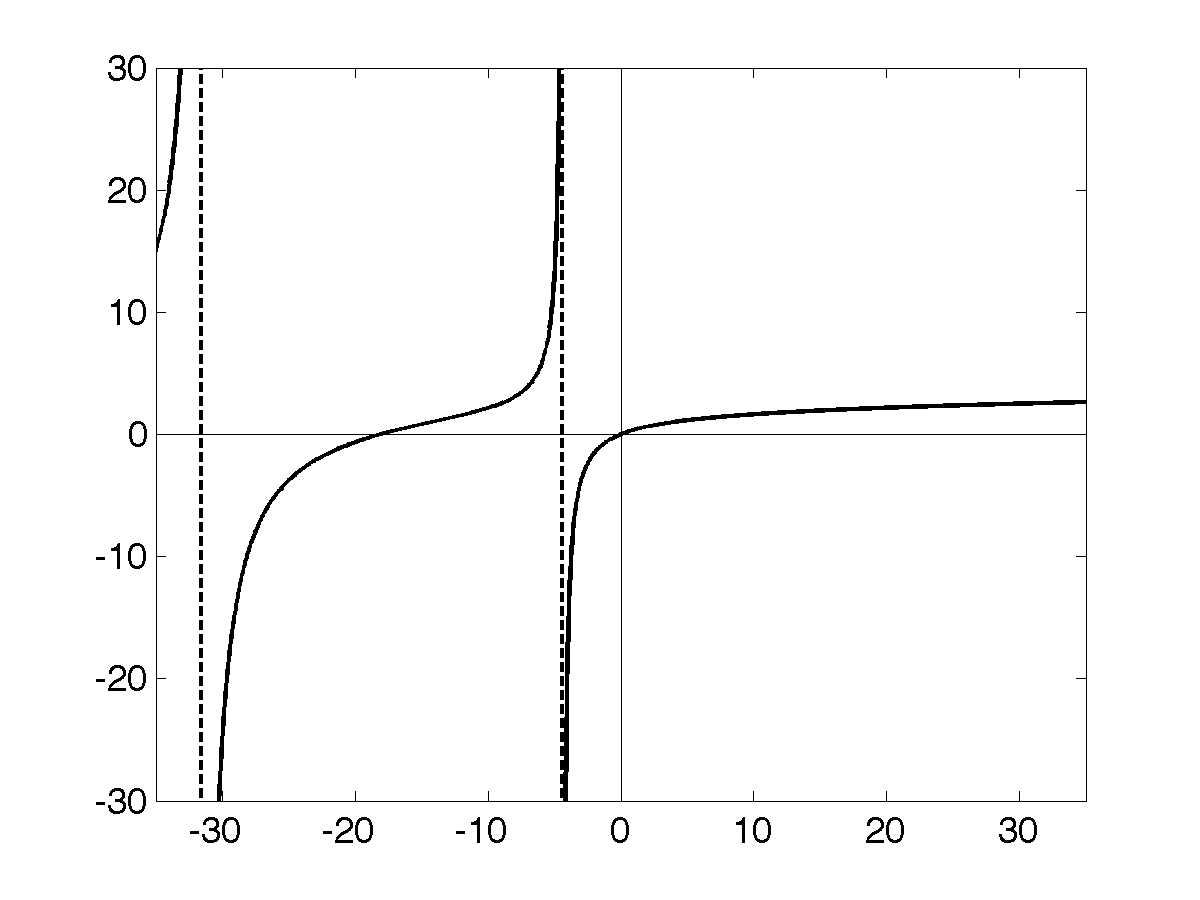}}
 \put(160,15){\includegraphics[width=7cm]{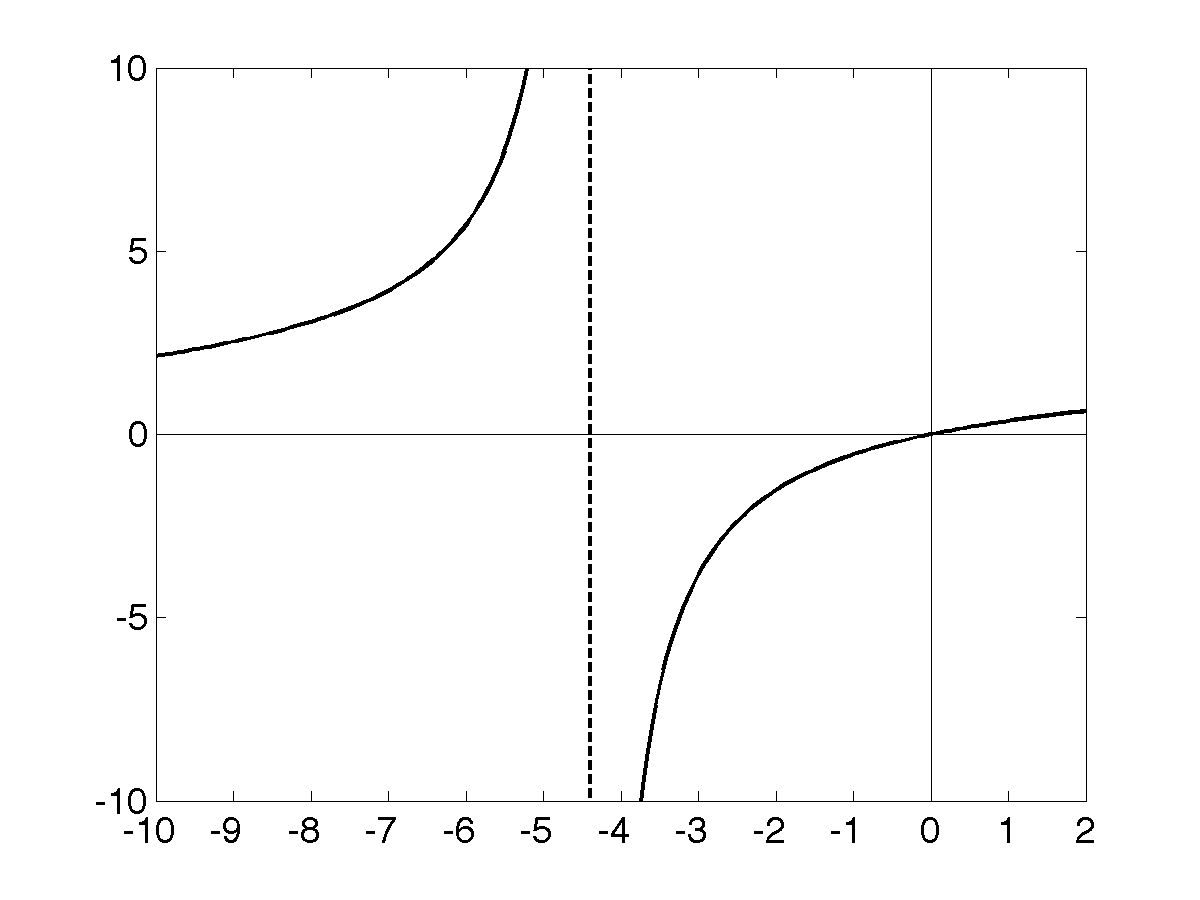}}
\put(42,0){$\lambda$}
\put(262,0){$\lambda$}
\end{picture}
\caption{\label{fig:qzeroevals2}Eigenvalues $\mu(\lambda)=\mu_0(\widetilde{q}_\lambda)$ corresponding to the second example potential. Left: plot with full parameter range $-35\leq \lambda\leq 35$. Right: detail of the left plot. Note that $\mu(0)=0$ and $\mu^\prime(0)>0$ as predicted by Lemma \ref{lemma:mu}. Also, note that Dirichlet eigenvalues of the potential $\widetilde{q}_\lambda$ in the unit disc cause singularities in $\mu(\lambda)$. Compare to Figure \ref{fig:qzeroevals}.}
\end{figure}

\begin{figure}
\begin{picture}(300,325)
 \put(-60,20){\includegraphics[width=14cm]{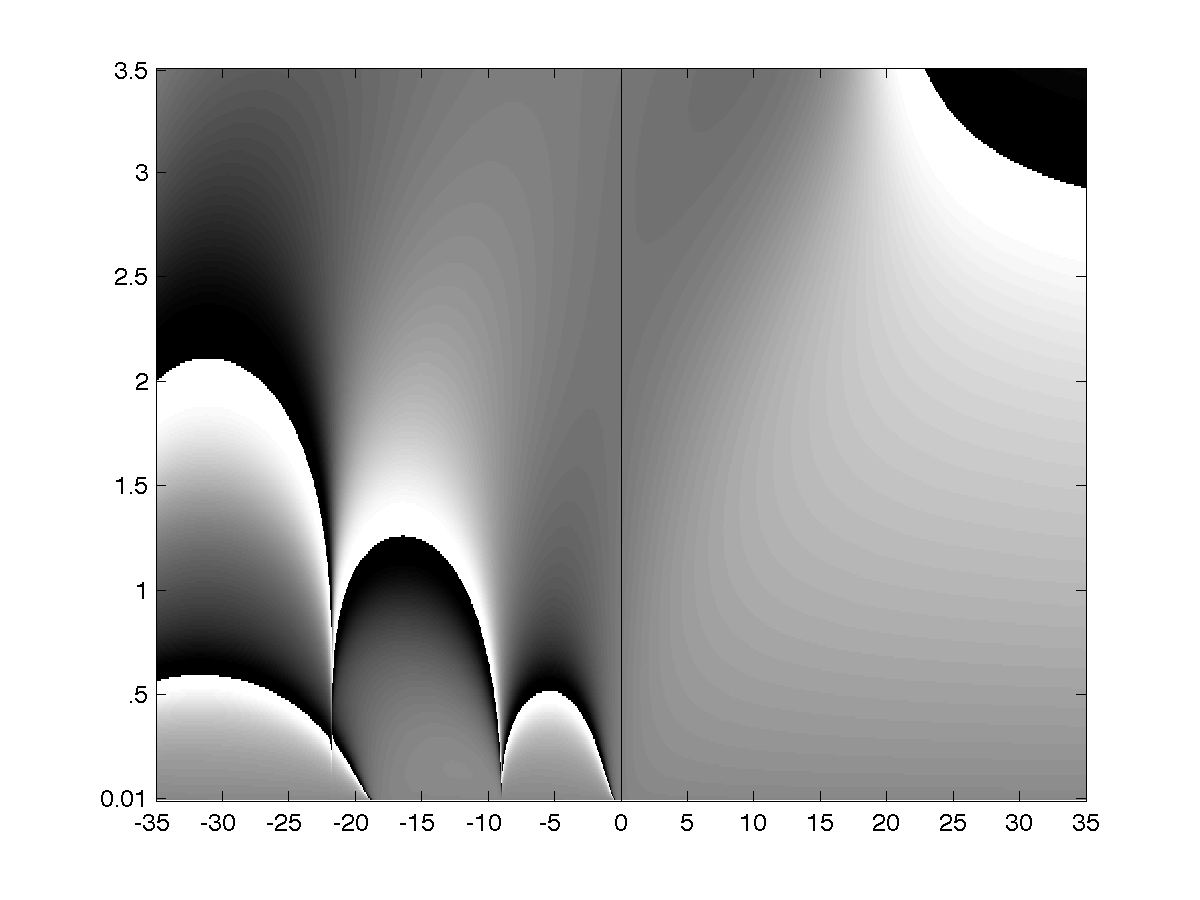}}
\put(140,0){$\lambda$}
\put(-75,200){\Large $|k|$}
\end{picture}
\caption{\label{fig:scat2A}Scattering transform corresponding to the second example. The horizontal axis is the parameter $\lambda$ in the definition ${\widetilde{q}}_\lambda = {\widetilde{q}}_0+\lambda w$ of the potential. The vertical axis is $|k|$.  There are curves along which a singular jump ``from $-\infty$ to $+\infty$'' appears.
The $k$ values at those curves are exceptional points. Compare to Figure \ref{fig:scat1A}.}
\end{figure}

\begin{figure}
\begin{picture}(400,160)
 \put(-20,10){\includegraphics[height=4.6cm]{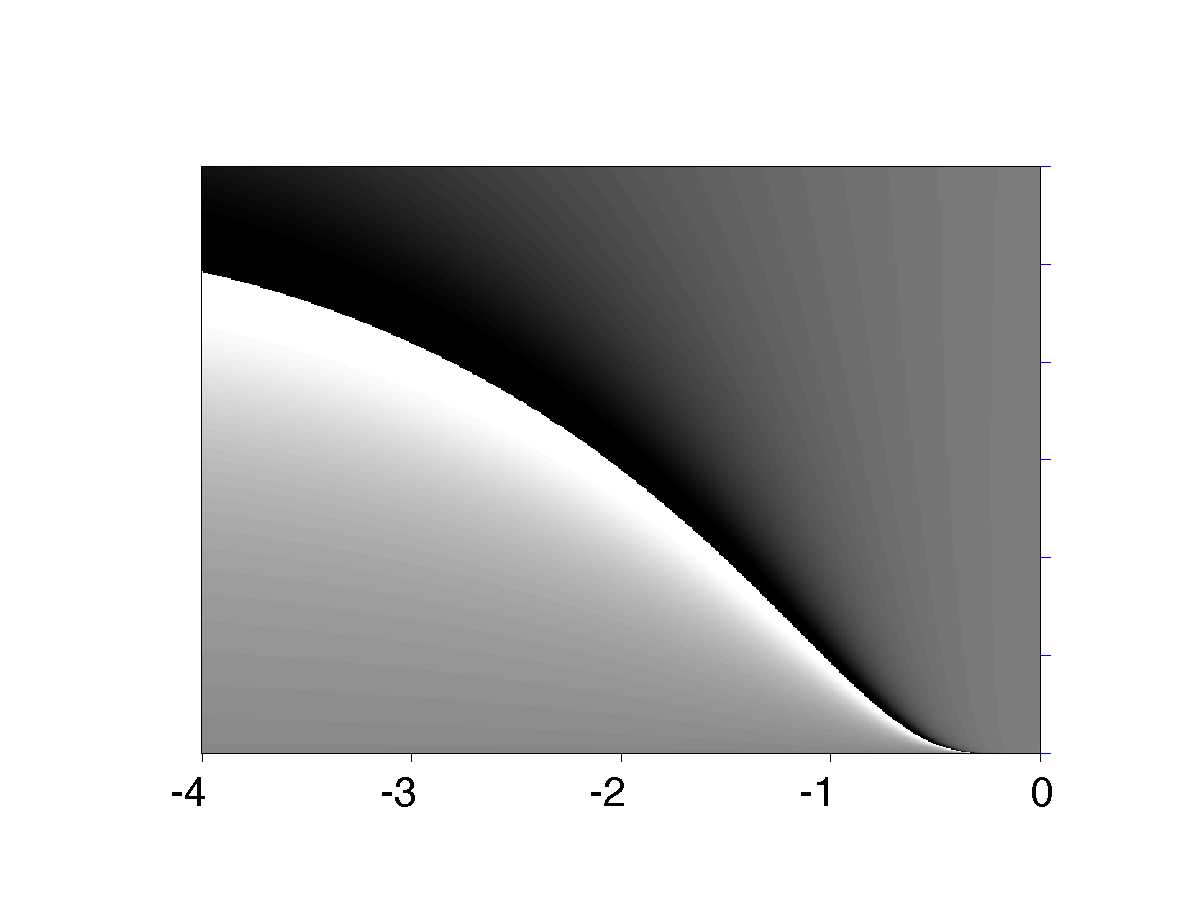}}
 \put(175,10){\includegraphics[height=4.8cm]{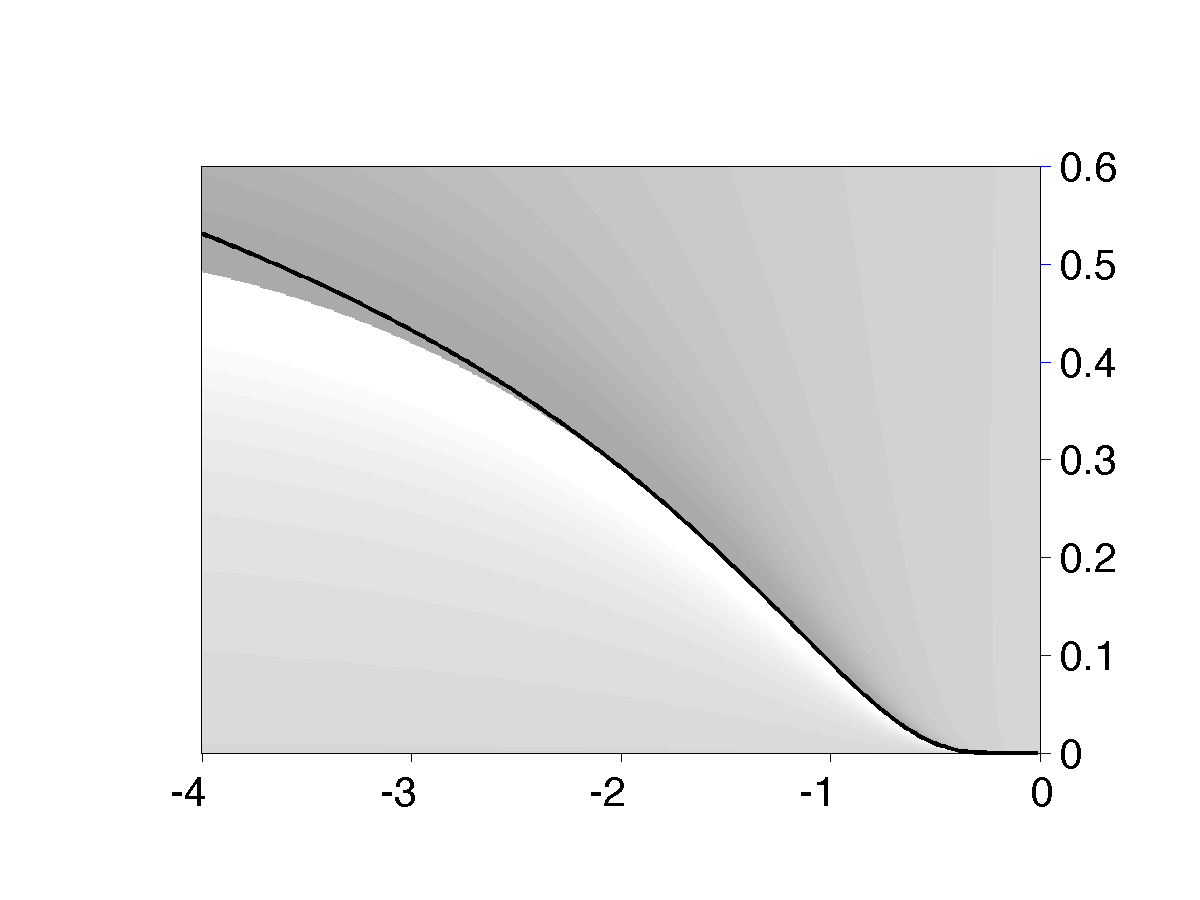}}
 \put(377,87){\large $|k|$}
 \put(105,0){$\lambda$}
 \put(300,0){$\lambda$}
 \put(220,140){$\displaystyle \exp\left[  2\pi\left(  h+\frac{1}{2\pi\mu(\lambda)}\right)  \right] $}
\end{picture}
\caption{\label{fig:asympfit2}Comparison of numerical results and the asymptotic formula (\ref{eq:r.lambda}) for the radius of the exceptional circle. This plot is for Example 2. Left: detail from Figure \ref{fig:scat2A} with parameters ranging in the rectangle $-4\leq\lambda\leq 0$ and $0.001\leq |k|\leq 0.6$. Right: the asymptotic function $r(\lambda)$ given by Theorem \ref{thm:main}. For ease of comparison, we also show in the background the pixel image from the left but with a lighter colormap. The asymptotic formula matches the computational result very closely in the interval $-2\leq\lambda\leq 0$. Compare to Figure \ref{fig:asympfit}.}
\end{figure}

\clearpage

\section{Discussion}

\noindent
We study zero-energy exceptional points of radial, compactly supported Schr\"odinger potentials in dimension two. Our work was inspired by  preliminary numerical experiments showing the emergence of singularities in the scattering transforms of potentials constructed by subtracting a test function from a radial conductivity-type potential.

We prove new results for radial, real-valued potentials of the form $q_\lambda = q+\lambda w$, where $q$ is of conductivity type, the test function $w$ is non-negative, and $\lambda\in\mathbb{R}$. It turns out that for small positive $\lambda$ there are no exceptional points and for small negative $\lambda$ there is exactly one exceptional circle whose radius we can determine asymptotically. In our two computational examples the asymptotic formula is quite accurate in the rather substantial interval $-2\leq \lambda \leq 0$, see Figures \ref{fig:asympfit} and \ref{fig:asympfit2}. See Figure  \ref{fig:scat1Adiagram} for a diagram illustrating all currently known theoretical facts about two-dimensional exceptional points at zero energy.

Our numerical computations raise some further theoretical questions as well. Figures \ref{fig:scat1A} and \ref{fig:scat2A} suggest that for negative $\lambda$ far away from zero there are several exceptional circles. Also, for large positive $\lambda$ some exceptional points appear, but it is unclear where in the $(\lambda,|k|)$ plane they originate. Furthermore, there are some quite complicated features in the parameter interval $-23\leq\lambda\leq-18.5$ that are very different for our two examples. See Figure \ref{fig:comparison} for a zoom-in.

\begin{figure}
\begin{picture}(300,325)
 \put(-119,-20){\includegraphics[width=17.5cm]{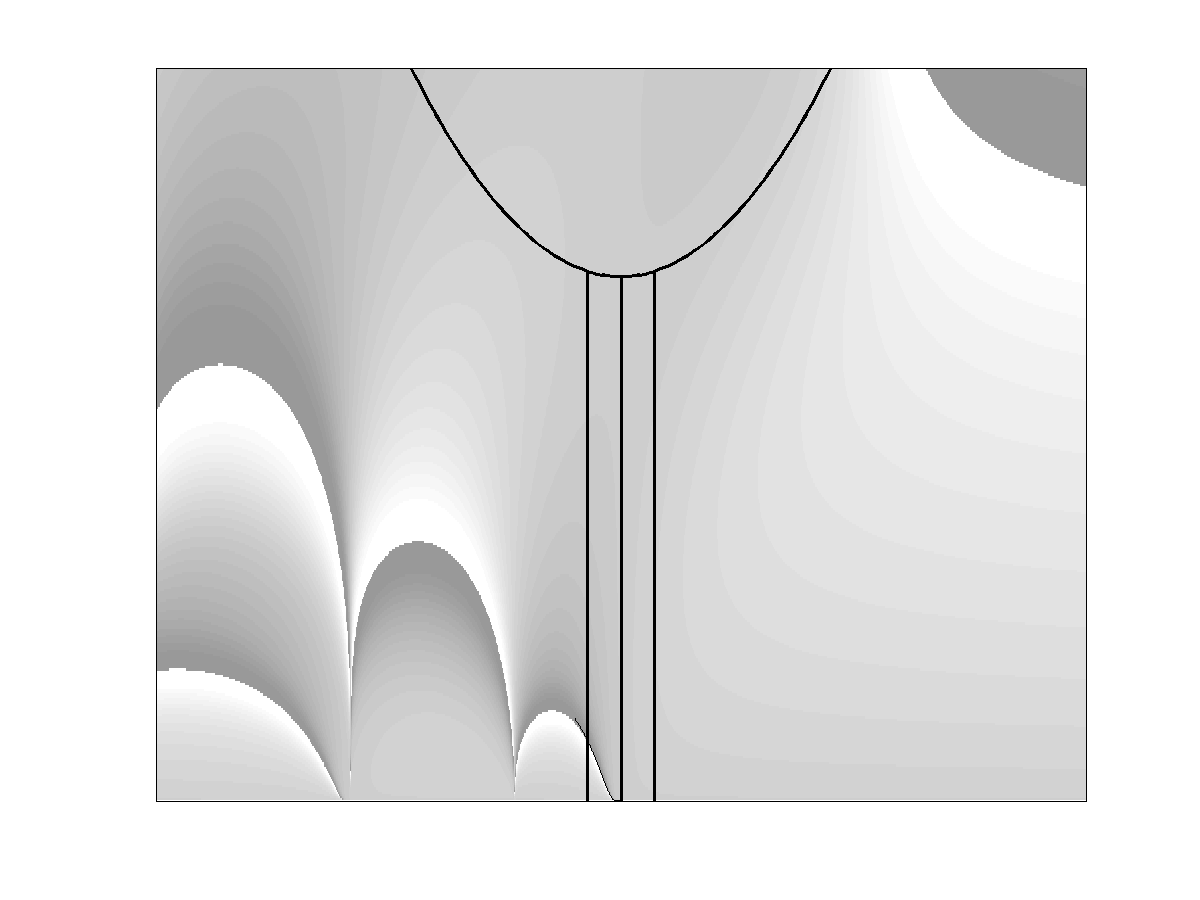}}
\put(136.5,8){$0$}
\put(150,8){$\lambda_0$}
\put(111,8){$-\lambda_0$}
\put(253,0){$\lambda$}
\put(-75,240){$|k|$}
\put(90,313){No exceptional points}
\put(127,65){\rotatebox{90}{Exactly one exceptional circle}}
\put(142,65){\rotatebox{90}{No nonzero exceptional points}}
\put(88,274){\rotatebox{-40}{$R(\lambda)$}}
\put(125.5,46){\rotatebox{-61}{\tiny $r(\lambda)$}}
\end{picture}
\caption{\label{fig:scat1Adiagram}Conceptual diagram illustrating the current understanding of exceptional points in dimension two. The underlying grayscale image shows the scattering transform of Figure \ref{fig:scat1A} corresponding to Example 1.  The potential $q_0$ is of conductivity type, and therefore by Nachman \cite{Nachman:1996} there are no exceptional points for $\lambda=0$. Theorem \ref{thm:main} above implies that there is a $\lambda_0>0$ such that (i) for $0<\lambda<\lambda_0$ the potential $q_\lambda$ has no nonzero exceptional points, and (ii) for $-\lambda_0<\lambda<0$ the complex numbers satisfying $|k|=r(\lambda)$ are the only nonzero exceptional points for $q_\lambda$.  Furthermore, the Neumann series argument in the proof of \cite[Thm 1.1]{Nachman:1996} shows that there exists a radius $R(\lambda)$, depending on $\|q_\lambda\|_{L^p(\mathbb{R}^2)}$, such that $q_\lambda$ does not have exceptional points satisfying $|k|>R(\lambda)$. Note that the number $\lambda_0$ and the radius $R(\lambda)$ shown above are not in scale. }
\end{figure}
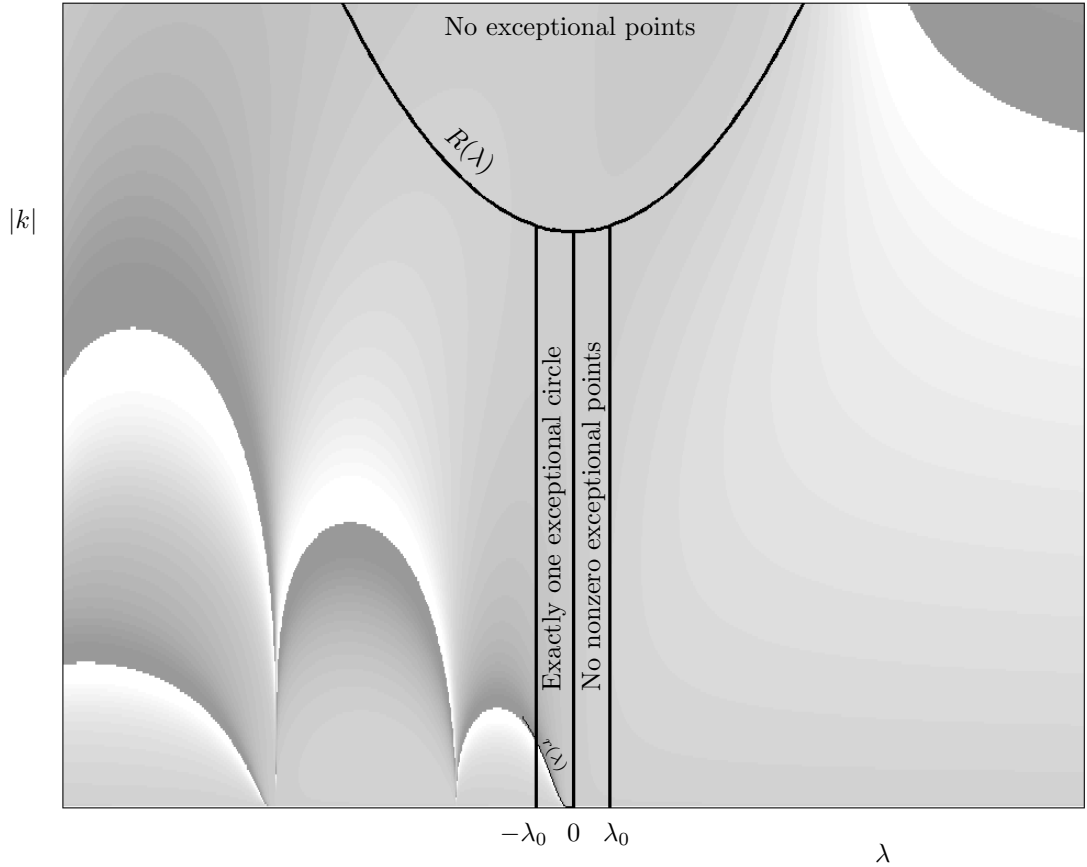

\begin{figure}
\begin{picture}(400,180)
 \put(20,15){\includegraphics[height=4.9cm]{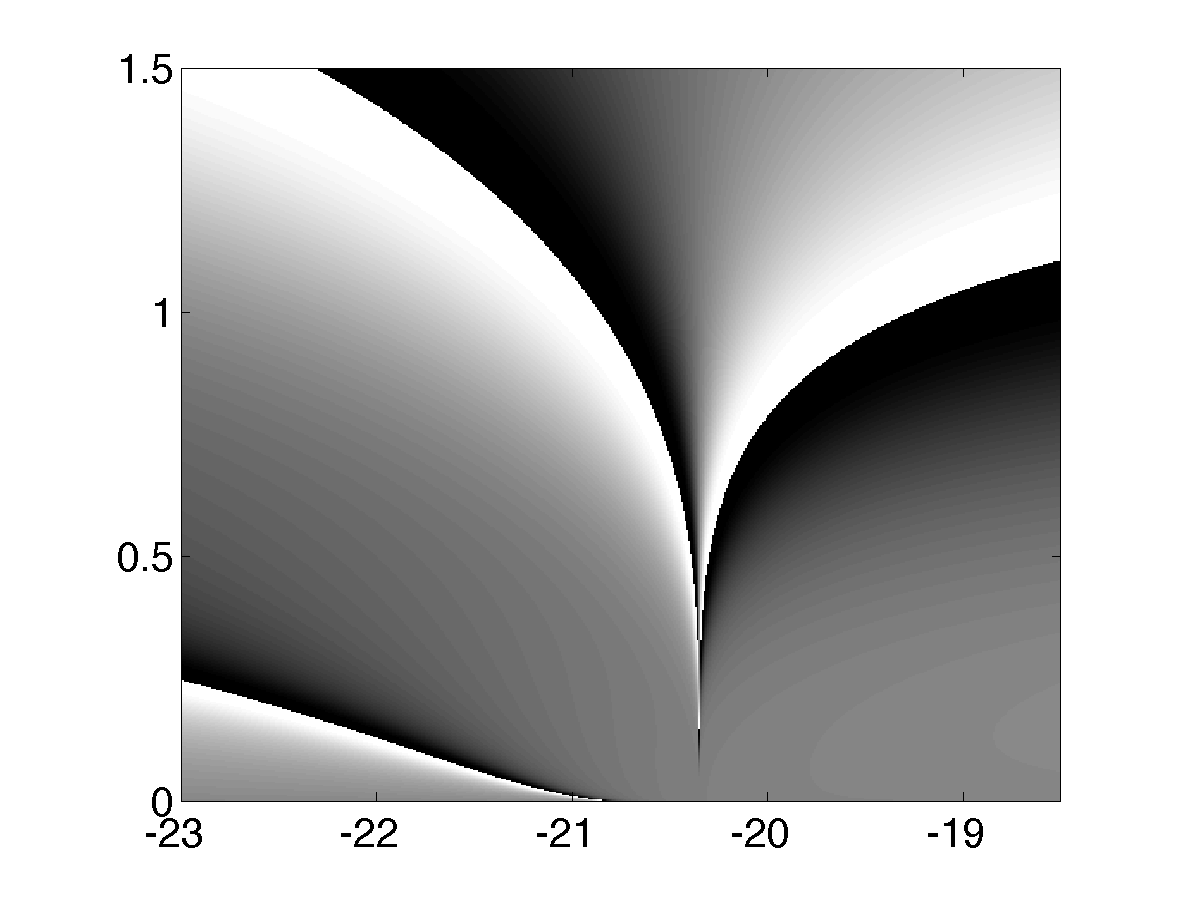}}
 \put(195,15){\includegraphics[height=4.8cm]{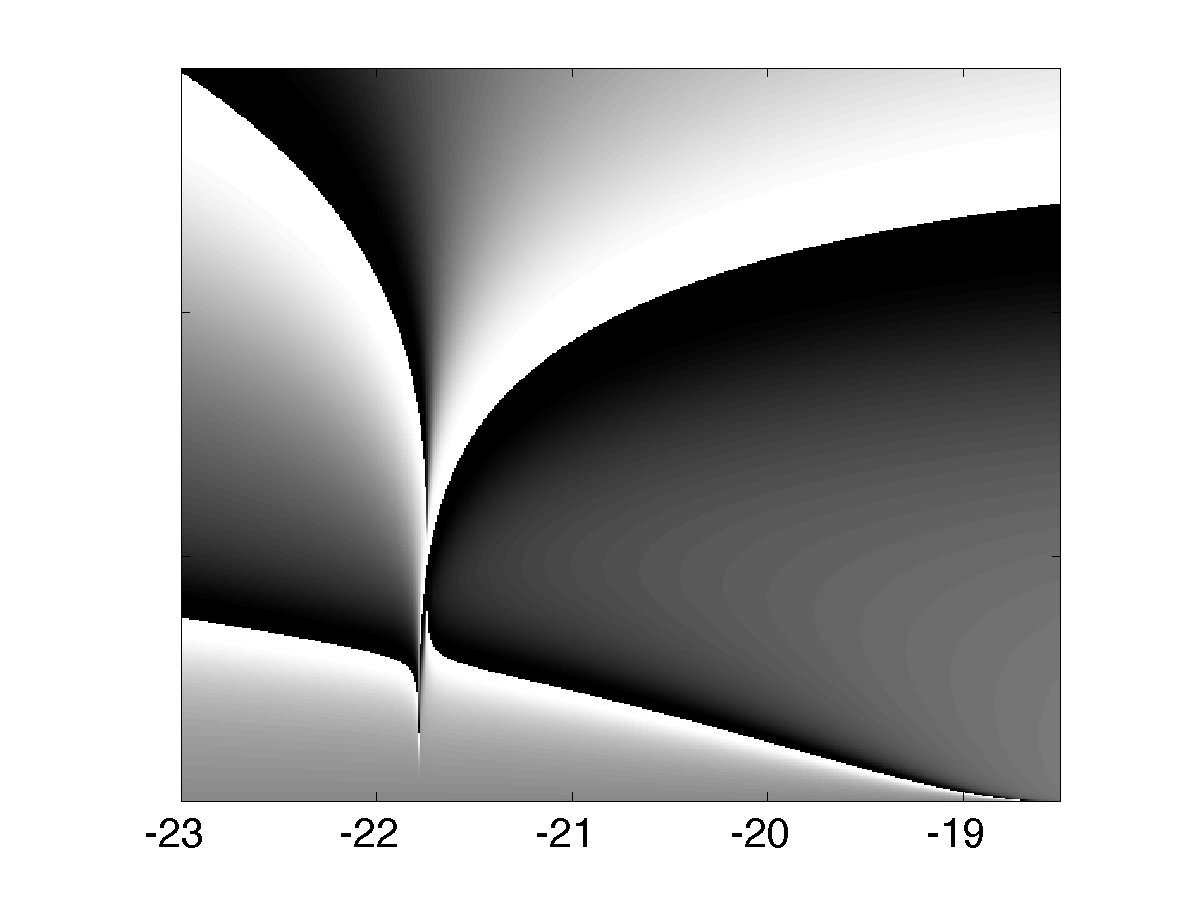}}
 \put(0,87){\large $|k|$}
 \put(105,0){\large $\lambda$}
 \put(278,0){\large $\lambda$}
 \put(29,158){\large Example 1}
 \put(203,158){\large Example 2}
\end{picture}
\caption{\label{fig:comparison}Comparison of the scattering transforms of the two example potentials in the parameter domain $-23\leq\lambda\leq-18.5$ and $0<|k|\leq 1.5$.}
\end{figure}

\appendix

\section{Radial symmetry of scattering transforms}\label{app:radial}

\noindent
Let $q_\lambda$ be a potential of the form (\ref{eq:q.lambda}). If $k\in\mathbb{C}\setminus 0$ is not an exceptional point of $q_\lambda$, then uniqueness of the CGO solution $\psi(z,k)$ shows that all $k^\prime\in\mathbb{C}$ with $|k^\prime|=|k|$ are non-exceptional for $q_\lambda$ as well. Furthermore, we can argue as in \cite[Section 4.1]{LMSS:2011} and find out that the scattering transform satisfies $\mathbf{t}_\lambda(k)=\mathbf{t}_\lambda(|k|)$ and $\overline{\mathbf{t}_\lambda(k)}=\mathbf{t}_\lambda(k)$.

Also, we know from Remark \ref{remark:circle} that exceptional points appear on a set of measure zero (circles centered at the origin). So, for illustrating $\mathbf{t}_\lambda(k)$ it is enough to display real-valued profiles of $\mathbf{t}_\lambda(|k|)$ evaluated at $|k|>0$, as is done in Figures \ref{fig:scat1A} and \ref{fig:scat2A}.

\section{Conductivity-type potentials and criticality}\label{appendix:Murata}

\noindent
Let $q(z)=q(|z|)$ be a radial potential of conductivity type in the sense of Definition \ref{def:condtype}. Then we can write $q=\psi^{-1}\left(  \Delta\psi\right)$ with some smooth and strictly positive function $\psi$ for which $\psi-1$ is compactly supported. Set
$$
  q_\lambda=q+\lambda w
$$
with $w\in C_0^\infty(\mathbb{R}^2)$ a radial, nonnegative test function which is not identically zero.

First of all, \cite[Theorem 3.1(iii)]{Murata:1986} implies that $q_0$ is critical. To see this, note that
$
  g_0(|z|) = \psi(|z|)
$
is the unique solution of \cite[equation (3.3)]{Murata:1986} with $j=0$ and $n=2$ satisfying the asymptotic condition $g_0(|z|)=1+o(1)$ as $|z|\rightarrow\infty$. Then the integral in \cite[formula (3.5)]{Murata:1986} diverges.

It follows from \cite[Theorem 2.4(i)]{Murata:1986} that the potential $q_\lambda$ is supercritical for all $\lambda<0$. In that case we know from \cite{alleg1974,MP: 1978,alleg1981} that there is no positive solution of $\left(  -\Delta+q_\lambda\right)  \psi=0$, so $q_\lambda$ is not of conductivity type.

It follows from \cite[Theorem 2.5(i)]{Murata:1986} that the potential $q_\lambda$ is subcritical for all $\lambda>0$. Furthermore, by \cite[Theorem 5.6(i)]{Murata:1986} the unique positive solution of $\left(  -\Delta+q_\lambda\right)  \psi=0$ has asymptotics $c\log|z|+O(1)$ with a positive constant $c$ as $|z|\rightarrow\infty$. Thus $q_\lambda$ is not of conductivity type because that would require a finite and constant limit $\lim_{\left\vert z\right\vert \rightarrow \infty}\psi(z)$.

\section{Spectral Properties of Conductivity-Type Potentials}
\label{appendix:spectrum}

The purpose of this appendix is to analyze the spectrum of the self-adjoint operator $H=-\Delta+q$ when
$q$ is a potential of conductivity type. 

\begin{proposition}
\label{prop:cond.spec}Suppose that $q_{0}\in\mathcal{C}_{0}^{\infty
}(\mathbb{R}^{2})$ is a potential of conductivity type, and let $H=-\Delta
+q_{0}$. Then $H$ has no $L^{2}$-eigenvalues.
\end{proposition}

\begin{proof}
By standard arguments (see, for example, Theorem XIII.56 in \cite{RSIV:1978}),
the equation $H\psi=\lambda\psi$ has no $L^{2}$ solutions for any $\lambda>0$,
while the positivity of the quadratic form associated to $H$ shows that the
there can be no eigenvalues with $\lambda<0$. It remains to show that, also, there are no
solutions of $H\psi=0$ that vanish at infinity, hence no $L^{2}$-eigenvalues
at zero energy. This is an immediate consequence of arguments in Nachman
\cite{Nachman:1996} but we reproduce them for the reader's convenience.

Suppose that $h\in H^{1}(\mathbb{R}^{2})$ is a weak solution of $Lh=0$
\ Without loss we may assume that $h$ is real-valued. By elliptic regularity, $h$ is a bounded, $\mathcal{C}^\infty$ function.  Let%
\[
v=h\partial\psi_{0}-\psi_{0}\partial h.
\]
Note that, as $\psi_{0}\in L^{\infty}$, $\partial\psi_{0}\in \mathcal{C}^\infty_0$, and $h\in
H^{1}$, it follows that $v\in L^{2}$. A straightforward computation using the
facts that $4\partial\overline{\partial}\psi_{0}=q\psi_{0}$ and $4\partial
\overline{\partial}h=qh$ shows $\overline{\partial}v=av-\overline{a}%
\overline{v}$ where $a=\overline{\partial}\psi_{0}/\psi_{0}$. Note that $a\in
\mathcal{C}^\infty_0$. By a standard vanishing theorem for generalized analytic functions (see for
example \cite{Vekua:1962} , we conclude that $v=0$, and hence that $\partial\left(  h/\psi
_{0}\right)  =0$, i,.e., $h/\psi_{0}$ is antiholomorphic. Since $h$ vanishes
at infinity while $\psi_{0}$ is bounded below, we conclude that $h=c\psi_{0}$
for a constant $c$, and hence, $h=0$ since $\psi_{0}(z)\rightarrow1$ as
$\left\vert z\right\vert \rightarrow\infty$.
\end{proof}

\begin{remark}
\label{rem:cond.unique}The argument above shows that a conductivity-type
potential is represented by a \emph{unique} normalized positive solution
$\psi_{0}$.
\end{remark}

Next, we consider spectral properties of $-\Delta+q$ on a bounded domain
containing the support of $q$. Recall that $H_{0}^{1}(\Omega)$ is the completion of $\mathcal{C}_0^\infty(\Omega)$ in the $H^1$-norm.

\begin{proposition}
\label{prop:cond.no.zero.ev}Let $\Omega$ be a bounded domain in $\mathbb{R}%
^{2}$ with smooth boundary, so chosen that $\operatorname*{supp}q_{0}$ is
strictly contained in $\Omega$. Denote by $H_{\Omega}$ the operator
$-\Delta+q_{0}$ with Dirichlet conditions on $\partial\Omega$. Then $0$ is not
a Dirichlet eigenvalue of $H_{\Omega\text{ }}$.
\end{proposition}

\begin{proof}
Suppose that $\psi\in H_{0}^{1}(\Omega)$ satisfies $H_{\Omega}\psi=0$. Extend
$\psi$ to a function $\chi$ in $H^{1}(\mathbb{R}^{2})$ by setting $\chi(z)=0$
for $z\in\mathbb{R}^{2}\backslash\Omega$. The distribution gradient of $\chi$
is given by $\left(  \nabla\chi\right)  (z)=(\nabla\psi)(z)$ for $z\in\Omega$,
and $(\nabla\chi)(z)=0$ otherwise. Letting $\mathfrak{q}$ be the quadratic
form of $-\Delta+q_{0}$ on $H^{1}(\mathbb{R}^{2})$, we have%
\[
\mathfrak{q}\left(  \chi,\chi\right)  =0
\]
so that, for any $\varphi\in\mathcal{C}_{0}^{\infty}(\mathbb{R}^{2})$, the
function%
\[
F(t)=\mathfrak{q}\left(  \chi+t\varphi,\chi+t\varphi\right)
\]
has an absolute minimum at $t=0$. Since $F^{\prime}(0)=0$ we recover
$\operatorname{Re}\mathfrak{q}(\varphi,\chi)=0$ for any such $\varphi$. It
follows that $\mathfrak{q}(\varphi,\chi)=0$ for all $\varphi\in\mathcal{C}%
_{0}^{\infty}(\mathbb{R}^{2})$, hence $\chi$ is a weak solution of $H\chi=0$
which vanishes identically outside $\Omega$. We can now use unique
continuation arguments (see, for example, \cite{RSIV:1978}, Theorem XIII.57)
now show that $\chi=0$, hence $\psi=0$.
\end{proof}

\end{document}